\documentclass[reqno]{amsart}

\usepackage{amsthm}
\usepackage{amsmath}
\usepackage[all]{xy} \SelectTips{eu}{}
\usepackage{latexsym,amsfonts,amssymb}
\usepackage{hyperref}
\usepackage{mathptmx}
\usepackage{nicefrac}
\usepackage{array}
\usepackage{xcolor}

\newcommand{\numberseries}{\bfseries}   

\newlength{\thmtopspace}                
\newlength{\thmbotspace}                
\newlength{\thmheadspace}               
\newlength{\thmindent}                  

\newtheoremstyle{fixed bf head,slanted body}
                {\thmtopspace}{\thmbotspace}{\slshape}
                {\thmindent}{\bfseries}{.}{\thmheadspace}
                {{\numberseries \thmnumber{#2\;}}\thmname{#1}\thmnote{ (#3)}}

\newtheoremstyle{variable bf head,slanted body}
                {\thmtopspace}{\thmbotspace}{\slshape}
                {\thmindent}{\bfseries}{.}{\thmheadspace}
                {{\numberseries \thmnumber{#2\;}}\thmname{#1}\thmnote{ #3}}

\newtheoremstyle{fixed bf head,upright body}
                {\thmtopspace}{\thmbotspace}{\upshape}
                {\thmindent}{\bfseries}{.}{\thmheadspace}
                {{\numberseries \thmnumber{#2\;}}\thmname{#1}\thmnote{ (#3)}}

\newtheoremstyle{numbered paragraph}
                {\thmtopspace}{\thmbotspace}{\upshape}
                {\thmindent}{\upshape}{}{\thmheadspace}
                {{\numberseries \thmnumber{#2.}}}

\theoremstyle{fixed bf head,slanted body}
\newtheorem{res}{}[section]
\newtheorem{thm}[res]{Theorem}          \newtheorem*{thm*}{Theorem}
\newtheorem{prp}[res]{Proposition}      \newtheorem*{prp*}{Proposition}
\newtheorem{cor}[res]{Corollary}        \newtheorem*{cor*}{Corollary}
\newtheorem{lem}[res]{Lemma}            \newtheorem*{lem*}{Lemma}

\theoremstyle{variable bf head,slanted body}
     \newtheorem*{introthm*}{Theorem}
   \newtheorem*{introcor*}{Corollary}

\theoremstyle{fixed bf head,upright body}
\newtheorem{stp}[res]{Setup}            \newtheorem*{stp*}{Setup}
\newtheorem{dfn}[res]{Definition}       \newtheorem*{dfn*}{Definition}
\newtheorem{con}[res]{Construction}     \newtheorem*{con*}{Construction}
\newtheorem{obs}[res]{Observation}      \newtheorem*{obs*}{Observation}
\newtheorem{rmk}[res]{Remark}           \newtheorem*{rmk*}{Remark}
\newtheorem{exa}[res]{Example}          \newtheorem*{exa*}{Example}
         \newtheorem*{qst*}{Question}

\theoremstyle{numbered paragraph}
\newtheorem{ipg}[res]{}

\newlength{\thmlistleft}        
\newlength{\thmlistright}       
\newlength{\thmlistpartopsep}   
\newlength{\thmlisttopsep}      
\newlength{\thmlistparsep}      
\newlength{\thmlistitemsep}     

\newcounter{eqc} 
  {\end{list}}%

\newcounter{prt}
\newenvironment{prt}{\begin{list}{\upshape (\alph{prt})}%
    {\usecounter{prt}%
      \setlength{\leftmargin}{\thmlistleft}%
      \setlength{\labelwidth}{\thmlistleft}%
      \setlength{\rightmargin}{\thmlistright}%
      \setlength{\partopsep}{\thmlistpartopsep}%
      \setlength{\topsep}{\thmlisttopsep}%
      \setlength{\parsep}{\thmlistparsep}%
      \setlength{\itemsep}{\thmlistitemsep}}}%
  {\end{list}}%

  \newcommand{\proofoftag}[2][:]{(#2)#1}

\newcommand{\pgref}[1]{\ref{#1}}
\newcommand{\thmref}[2][Theorem~]{#1\pgref{thm:#2}}
\newcommand{\corref}[2][Corollary~]{#1\pgref{cor:#2}}
\newcommand{\prpref}[2][Proposition~]{#1\pgref{prp:#2}}
\newcommand{\lemref}[2][Lemma~]{#1\pgref{lem:#2}}

\newcommand{\dfnref}[2][Definition~]{#1\pgref{dfn:#2}}
\newcommand{\exaref}[2][Example~]{#1\pgref{exa:#2}}
\newcommand{\rmkref}[2][Remark~]{#1\pgref{rmk:#2}}
\newcommand{\secref}[2][Section~]{#1\ref{sec:#2}}

\renewcommand{\eqref}[1]{(\pgref{eq:#1})}

\newcommand{\stpref}[2][Setup~]{#1\pgref{stp:#2}}

\makeatletter
\def\@nobreak@#1{\mathchoice%
  {\nobreakdef@\displaystyle\f@size{#1}}%
  {\nobreakdef@\nobreakstyle\tf@size{\firstchoice@false #1}}%
  {\nobreakdef@\nobreakstyle\sf@size{\firstchoice@false #1}}%
  {\nobreakdef@\nobreakstyle\ssf@size{\firstchoice@false #1}}%
  \check@mathfonts}%
\def\nobreakdef@#1#2#3{\hbox{{%
                    \everymath{#1}%
                    \let\f@size#2\selectfont%
                    #3}}}%
\makeatother

\DeclareFontFamily{T1}{cmex}{}
\DeclareFontShape{T1}{cmex}{m}{n}{<-> s * [0.89] cmex10}{}
\DeclareSymbolFont{cmlargesymbols}{T1}{cmex}{m}{n}
\DeclareMathSymbol{\mycoprod}{\mathop}{cmlargesymbols}{"60} 
\DeclareMathSymbol{\myprod}{\mathop}{cmlargesymbols}{"51} \let\prod\myprod

\DeclareSymbolFont{usualmathcal}{OMS}{cmsy}{m}{n}
\DeclareSymbolFontAlphabet{\mathcal}{usualmathcal}

\DeclareSymbolFont{letters}{OML}{txmi}{m}{it}

\DeclareMathSymbol{\alpha}{\mathord}{letters}{"0B}
\DeclareMathSymbol{\beta}{\mathord}{letters}{"0C}
\DeclareMathSymbol{\gamma}{\mathord}{letters}{"0D}
\DeclareMathSymbol{\delta}{\mathord}{letters}{"0E}
\DeclareMathSymbol{\epsilon}{\mathord}{letters}{"0F}
\DeclareMathSymbol{\zeta}{\mathord}{letters}{"10}
\DeclareMathSymbol{\eta}{\mathord}{letters}{"11}
\DeclareMathSymbol{\theta}{\mathord}{letters}{"12}
\DeclareMathSymbol{\iota}{\mathord}{letters}{"13}
\DeclareMathSymbol{\kappa}{\mathord}{letters}{"14}
\DeclareMathSymbol{\lambda}{\mathord}{letters}{"15}
\DeclareMathSymbol{\mu}{\mathord}{letters}{"16}
\DeclareMathSymbol{\nu}{\mathord}{letters}{"17}
\DeclareMathSymbol{\xi}{\mathord}{letters}{"18}
\DeclareMathSymbol{\pi}{\mathord}{letters}{"19}
\DeclareMathSymbol{\rho}{\mathord}{letters}{"1A}
\DeclareMathSymbol{\sigma}{\mathord}{letters}{"1B}
\DeclareMathSymbol{\tau}{\mathord}{letters}{"1C}
\DeclareMathSymbol{\upsilon}{\mathord}{letters}{"1D}
\DeclareMathSymbol{\phi}{\mathord}{letters}{"1E}
\DeclareMathSymbol{\chi}{\mathord}{letters}{"1F}
\DeclareMathSymbol{\psi}{\mathord}{letters}{"20}
\DeclareMathSymbol{\omega}{\mathord}{letters}{"21}
\DeclareMathSymbol{\varepsilon}{\mathord}{letters}{"22}
\DeclareMathSymbol{\vartheta}{\mathord}{letters}{"23}
\DeclareMathSymbol{\varpi}{\mathord}{letters}{"24}
\DeclareMathSymbol{\varrho}{\mathord}{letters}{"25}
\DeclareMathSymbol{\varsigma}{\mathord}{letters}{"26}
\DeclareMathSymbol{\varphi}{\mathord}{letters}{"27}

\DeclareMathSymbol{\Gamma}{\mathord}{letters}{"00}
\DeclareMathSymbol{\Delta}{\mathord}{letters}{"01}
\DeclareMathSymbol{\Theta}{\mathord}{letters}{"02}
\DeclareMathSymbol{\Lambda}{\mathord}{letters}{"03}
\DeclareMathSymbol{\Xi}{\mathord}{letters}{"04}
\DeclareMathSymbol{\Pi}{\mathord}{letters}{"05}
\DeclareMathSymbol{\Sigma}{\mathord}{letters}{"06}
\DeclareMathSymbol{\Upsilon}{\mathord}{letters}{"07}
\DeclareMathSymbol{\Phi}{\mathord}{letters}{"08}
\DeclareMathSymbol{\Psi}{\mathord}{letters}{"09}
\DeclareMathSymbol{\Omega}{\mathord}{letters}{"0A}

\DeclareMathSymbol{\upGamma}{\mathalpha}{operators}{"00}
\DeclareMathSymbol{\upDelta}{\mathalpha}{operators}{"01}
\DeclareMathSymbol{\upTheta}{\mathalpha}{operators}{"02}
\DeclareMathSymbol{\upLambda}{\mathalpha}{operators}{"03}
\DeclareMathSymbol{\upXi}{\mathalpha}{operators}{"04}
\DeclareMathSymbol{\upPi}{\mathalpha}{operators}{"05}
\DeclareMathSymbol{\upSigma}{\mathalpha}{operators}{"06}
\DeclareMathSymbol{\upUpsilon}{\mathalpha}{operators}{"07}
\DeclareMathSymbol{\upPhi}{\mathalpha}{operators}{"08}
\DeclareMathSymbol{\upPsi}{\mathalpha}{operators}{"09}
\DeclareMathSymbol{\upOmega}{\mathalpha}{operators}{"0A}

\renewcommand{\con}[1]{\textnormal{({\small #1})}}

\newcommand{\cA}{\mathcal{A}}
\newcommand{\cB}{\mathcal{B}}
\newcommand{\cC}{\mathcal{C}}

\newcommand{\Fu}[1]{\mathrm{#1}}
\newcommand{\Ld}[2]{\mathrm{L}_{#1}#2}
\newcommand{\Rd}[2]{\mathrm{R}^{#1}#2}

\newcommand{\Fix}[3][(\Fu{F},\Fu{G})]{\mathsf{Fix}^{#1}_{#2}(#3)}
\newcommand{\coFix}[3][(\Fu{F},\Fu{G})]{\mathsf{coFix}^{#1}_{#2}(#3)}

\renewcommand{\H}[2]{\mathrm{H}_{#1}(#2)}
\newcommand{\Hnop}[2]{\mathrm{H}_{#1}#2}

\newcommand{\Mod}[1]{\mathsf{Mod}(#1)}
\renewcommand{\mod}[1]{\mathsf{mod}(#1)}

\newcommand{\K}[1]{\mathsf{K}(#1)}

\newcommand{\Ker}[1]{\operatorname{Ker}#1}
\newcommand{\Coker}[1]{\operatorname{Coker}#1}

\newcommand{\Lo}{\Lambda^\mathrm{o}}
\newcommand{\Go}{\Gamma^\mathrm{o}}

\newcommand{\Proj}[1]{\mathrm{Prj}(#1)}

\newcommand{\Inj}[1]{\mathrm{Inj}(#1)}

\newcommand{\add}[2][]{\mathrm{add}_{#1}(#2)}

\newcommand{\Ext}[4]{\operatorname{Ext}_{#1}^{\mspace{1mu}#2}(#3,#4)}
\newcommand{\Tor}[4]{\operatorname{Tor}^{\mspace{2mu}#1}_{#2}(#3,#4)}
\newcommand{\Hom}[3]{\operatorname{Hom}_{#1}(#2,#3)}

\renewcommand{\L}[2][\mathfrak{a}]{\upLambda^{\mspace{-2mu}#1}#2}
\newcommand{\G}[2][\mathfrak{a}]{\upGamma_{\mspace{-2mu}#1}#2}  

\newcommand{\RHom}[3]{\mathbf{R}\mspace{-1mu}\operatorname{Hom}_{#1}(#2,#3)}
\newcommand{\LL}[2][\mathfrak{a}]{\mathbf{L}\upLambda^{\mspace{-2mu}#1}#2}
\newcommand{\RG}[2][\mathfrak{a}]{\mathbf{R}\upGamma_{\mspace{-2mu}#1}#2}

\newcommand{\pd}[2]{\mathrm{pd}_{#1}(#2)}

\newcommand{\id}[2]{\mathrm{id}_{#1}(#2)}

\newcommand{\D}[1]{\mathsf{D}(#1)}

\newcommand{\gen}[3]{\mathsf{gen}^{#1}_{#2}(#3)}
\newcommand{\cogen}[3]{\mathsf{cogen}_{#1}^{#2}(#3)}
\renewcommand{\res}[3]{\mathsf{res}^{#1}_{#2}(#3)}
\newcommand{\cores}[3]{\mathsf{cores}_{#1}^{#2}(#3)}

\newcommand{\C}[1][\mathfrak{a}]{\mathrm{C}(#1)}
\newcommand{\lc}[3][\mathfrak{a}]{\mathrm{H}_{#1}^{#2}(#3)}
\newcommand{\lce}[2][\mathfrak{a}]{\mathrm{H}_{#1}^{#2}}
\newcommand{\lh}[3][\mathfrak{a}]{\mathrm{H}^{#1}_{#2}(#3)}
\newcommand{\lhe}[2][\mathfrak{a}]{\mathrm{H}^{#1}_{#2}}

\newcommand{\CM}[3][\mathfrak{a}]{\mathsf{CM}_{#1}^{#2}(#3)}

\begin{document}

\vspace*{-0.1ex}

\title[Equivalences from tilting theory and commutative algebra]{Equivalences from tilting theory and commutative algebra from the adjoint functor point of view}
\author{Olgur Celikbas \ }
\address{Department of Mathematics, West Virginia University, Morgantown, WV 26506 U.S.A} 
\email{olgur.celikbas@math.wvu.edu}

\author{ \ Henrik Holm}
\address{Department of Mathematical Sciences, Universitetsparken 5, University of Co\-penhagen, 2100 Copenhagen {\O}, Denmark} 
\email{holm@math.ku.dk}
\urladdr{http://www.math.ku.dk/\~{}holm/}






\begin{abstract}
  We give a category theoretic approach to several known equivalences from (classic) tilting theory and commutative algebra. Furthermore, we apply our main results to establish a duality theory for relative Cohen--Macaulay modules in the sense of Hellus, Schenzel, and Zargar.
\end{abstract}

\maketitle


\section{Introduction}
\label{sec:Introduction}

In this paper, we consider an adjunction \mbox{$\Fu{F} \colon\mspace{-2mu} \cA \rightleftarrows \cB \colon \Fu{G}$} between abelian categories. Even though the pair $(\Ld{\ell}{\Fu{F}},\Rd{\ell}{\Fu{G}})$ of $\ell^\mathrm{\,th}$ (left/right) derived functors is generally not an adjunction $\cA \rightleftarrows \cB$, one can obtain an adjunction, and even an adjoint equivalence, from these functors by restricting them appropriately. More precisely, in \dfnref{Fix-Cofix} we introduce two subcategories $\Fix[]{\ell}{\cA}$, the category of \emph{$\ell$-fixed} objects in $\cA$, and $\coFix[]{\ell}{\cB}$, the category of \emph{$\ell$-cofixed} objects in $\cB$, and show in \thmref{equivalence} that 
one gets an adjoint equivalence: \vspace*{-0.25ex}
\begin{equation}
  \label{eq:intro1}
  \xymatrix@C=3pc{
    \Fix[]{\ell}{\cA}\, \ar@<0.6ex>[r]^-{\Ld{\ell}{\Fu{F}}} & \,\coFix[]{\ell}{\cB} \ar@<0.6ex>[l]^-{\Rd{\ell}{\Fu{G}}}
  }\!. \vspace*{-0.25ex}
\end{equation} 
When the adjunction $(\Fu{F},\Fu{G})$ is suitably nice---more precisely, when it is a \emph{tilting adjunction} in the sense of \dfnref{tilting-adjunction}---the adjoint equivalence \eqref{intro1} takes the simpler form: \vspace*{-0.25ex}
\begin{equation}
  \label{eq:intro2}
  \xymatrix@C=3pc{
  {
  \{ 
     A \in \cA \ | \ 
     \Ld{i}{\Fu{F}}(A)=0 \text{ for } i \neq \ell
  \mspace{1mu}\}
  }
  \ar@<0.6ex>[r]^-{\Ld{\ell}{\Fu{F}}}
  &
  {
  \{ 
     B \in \cB  \ | \
     \Rd{i}{\Fu{G}}(B)=0 \text{ for } i \neq \ell
  \mspace{1mu}\}
  }
  \ar@<0.6ex>[l]^-{\Rd{\ell}{\Fu{G}}}
  }\!,\vspace*{-0.25ex}
\end{equation}      
as shown in \thmref{equivalence-simplified}. These equivalences, which are our main results, are proved in \secref{Fixed}. In \secref{Applications} we apply them to various situations and recover a number of known results from tilting theory and commutative algebra, such as the Brenner--Butler~and Happel theorem \cite{BrennerButler, Happel}, Wa\-ka\-matsu's duality \cite{Wakamatsu04}, and Foxby equivalence \cite{LLAHBF97, HBF72}. Details~can be found in \corref[Corollaries~]{BBH}, \corref[]{Wakamatsu}, and \corref[]{Foxby}.

In \secref{Derivatives} we investigate the equivalence \eqref{intro1} further in the special case where $\ell=0$. Under suitable hypotheses, we show in \thmref{equivalence-gen}  that for any $X \in \Fix[]{0}{\cA}$ and $d \geqslant 0$, \eqref{intro1} restricts to an equivalence: \vspace*{-0.25ex}
\begin{equation}
  \label{eq:intro3}
  \xymatrix@C=3pc{
    \Fix[]{0}{\cA} \mspace{1mu}\cap\mspace{1mu} \gen{\cA}{d}{X}\, \ar@<0.6ex>[r]^-{\Fu{F}}  & \, \coFix[]{0}{\cB} \mspace{1mu}\cap\mspace{1mu} \gen{\cB}{d}{\Fu{F}X} \ar@<0.6ex>[l]^-{\Fu{G}}
  },\vspace*{-0.25ex}
\end{equation}  
where $\gen{\cA}{d}{X}$ is the full subcategory of $\cA$ consisting of objects that are finitely built from $X$ in the sense of \dfnref{gen-res}. Although \eqref{intro3}
looks more technical than \eqref{intro1} and \eqref{intro2}, it too has useful applications, for example, it contains as a special case Matlis'~duality~\cite{EMt58}:
\begin{displaymath} 
  \xymatrix@C=6pc{
    \{ \textnormal{Finitely generated $R$-modules} \}
    \ar@<0.6ex>[r]^-{\Hom{R}{-}{E_R(k)}} 
    & 
    \{ \textnormal{Artinian $R$-modules} \}
    \ar@<0.6ex>[l]^-{\Hom{R}{-}{E_R(k)}}    
  }\!,
\end{displaymath}
where $R$ is a commutative noetherian local complete ring; see \corref{Matlis}. \thmref{equivalence-res} is a variant of \eqref{intro3} which yields Sharp's equivalence \cite{RYS72} for finitely generated modules of finite projective/injective dimension over Cohen--Macaulay rings; see \corref{Sharp}.

In \secref{CM} we apply the equivalence \eqref{intro1} to study relative Cohen--Macaulay modules. To explain what this is about, recall that for a (non-zero) finitely generated module $M$ over a commutative noetherian local ring $(R,\mathfrak{m},k)$, which we assume is complete, one has
\begin{displaymath}
  \mathrm{depth}_RM = \min\{i\,|\,\lc[\mathfrak{m}]{i}{M}\neq 0\}
  \quad \text{ and } \quad
  \mathrm{dim}_RM = \max\{i\,|\,\lc[\mathfrak{m}]{i}{M}\neq 0\}\;,
\end{displaymath}
where $\lce[\mathfrak{m}]{i}$ denotes the $i^\mathrm{\,th}$ local cohomology module w.r.t.~$\mathfrak{m}$. Hence $M$ is Cohen--Macaulay (CM) of dimension $t$ if and only if $\lc[\mathfrak{m}]{i}{M}=0$ for $i \neq t$. When $R$ itself is CM, the most important and useful fact about the category of $t$-dimensional CM modules is the duality
\begin{displaymath}
  \xymatrix@C=3.5pc{
  \{M \in \mod{R} \,|\, \lc[\mathfrak{m}]{i}{M}=0 \textnormal{ for } i \neq t \}  
  \ar@<0.6ex>[r]^-{\Ext{R}{c-t}{-}{\upOmega}
  }
  &
  \{M \in \mod{R} \,|\, \lc[\mathfrak{m}]{i}{M}=0 \textnormal{ for } i \neq t \}
  \ar@<0.6ex>[l]^-{\Ext{R}{c-t}{-}{\upOmega}}
  }\!,
\end{displaymath}
where $c$ is the Krull dimension of $R$ and $\upOmega$ is the dualizing module. The theory of CM modules over CM rings is an active research area and in recent papers by e.g.~Hellus and Schenzel \cite{HellusSchenzel} and Zargar \cite{Zargar}, it was suggested to investigate this theory relative to an ideal $\mathfrak{a} \subset R$. That is, in the case where $R$ is \emph{relative CM w.r.t.~$\mathfrak{a}$}, meaning that $\lc{i}{R}=0$ for $i \neq c$ where $\mathrm{depth}_R(\mathfrak{a},R) = c = \mathrm{cd}_R(\mathfrak{a},R)$, one wishes to study the category
\begin{equation}
  \label{eq:intro4}
  \{M \in \mod{R} \,|\, \lc{i}{M}=0 \textnormal{ for } i \neq t \} \qquad \text{(for any $t$)}
\end{equation}
of finitely generated \emph{relative CM $R$-modules of cohomological dimension $t$ w.r.t.~$\mathfrak{a}$}. \mbox{Towards} a relative CM theory, the first thing one should start looking for is a duality on the category \eqref{intro4}. Unfortunately such a duality does not exist in general; indeed for $\mathfrak{a}=0$ (the~zero~ideal) and $t=0$ the category in \eqref{intro4} is the category $\mod{R}$ of all finitely generated $R$-modules, which is self-dual only in very special cases (if $R$ is Artinian). To fix this problem, we introduce in \dfnref{CM} another category, $\CM{t}{R}$, of (not necessarily finitely generated) $R$-modules; it is an extension of the category \eqref{intro4} in the sense that:
\begin{displaymath}
  \CM{t}{R} \,\cap\, \mod{R} \,=\, \{M \in \mod{R} \,|\, \lc{i}{M}=0 \textnormal{ for all } i \neq t \}\;.
\end{displaymath}
Our main result about this (larger) category is that it is self-dual. We show in \thmref{CM2} that if $R$ is relative CM w.r.t.~$\mathfrak{a}$ with $\mathrm{depth}_R(\mathfrak{a},R) = c = \mathrm{cd}_R(\mathfrak{a},R)$, then there is a duality:
\begin{equation}
  \label{eq:intro5}
  \xymatrix@C=5pc{
    \CM{t}{R} \ar@<0.6ex>[r]^-{\Ext{R}{c-t}{-}{\upOmega_\mathfrak{a}}} & \CM{t}{R} \ar@<0.6ex>[l]^-{\Ext{R}{c-t}{-}{\upOmega_\mathfrak{a}}}
  }\!,
\end{equation}
where $\upOmega_\mathfrak{a}$ is the module from \dfnref{Omega}. It is worth pointing out two extreme cases of this duality: For $\mathfrak{a}=\mathfrak{m}$ a ring is relative CM w.r.t.~$\mathfrak{a}$ if and only if it is CM in the ordinary sense, and in this case $c$ is the Krull dimension of $R$ and $\upOmega_\mathfrak{a}=\upOmega$ is a dualizing module; see \exaref{Omega}. Thus \eqref{intro5} extends the classic duality for CM modules of Krull dimension $t$ mentioned above. For $\mathfrak{a}=0$ any ring is relative CM w.r.t.~$\mathfrak{a}$, and \eqref{intro5} specializes, in view of \exaref[Examples~]{CM-0} and \exaref[]{Omega}, to the (well-known and almost trivial) duality:
\begin{displaymath} 
  \xymatrix@C=5.5pc{
    \{ \textnormal{Matlis reflexive $R$-modules} \}
    \ar@<0.6ex>[r]^-{\Hom{R}{-}{E_R(k)}} 
    & 
    \{ \textnormal{Matlis reflexive $R$-modules} \}
    \ar@<0.6ex>[l]^-{\Hom{R}{-}{E_R(k)}}    
  }\!.
\end{displaymath}
Hence \eqref{intro5} is a family of dualities, parameterized by ideals $\mathfrak{a} \subset R$, that connects the known dualities for (classic) CM modules and Matlis reflexive modules.

We end this introduction by explaining how our work is related to the literature: 

For $\ell=0$ the equivalence \eqref{intro1} follows from Frankild and J{\o}rgensen \cite[Thm.~(1.1)]{AFrPJr02} as $(\Ld{0}{\Fu{F}},\Rd{0}{\Fu{G}}) = (\Fu{F},\Fu{G})$ is an adjunction $\cA \rightleftarrows \cB$ to begin with. For  $\ell>0$ it requires some more work as the pair $(\Ld{\ell}{\Fu{F}},\Rd{\ell}{\Fu{G}})$ is not an adjunction. Nevertheless, having made the necessary preparations, 
the proof of the adjoint equivalence \eqref{intro1} is completely formal.

The idea of reproving and extending known equivalences/dualities from commutative algebra via an abstract approach, like we do, is certainly not new. In fact, this is the main idea in, for example, \cite{AFrPJr02,AFrPJr04} by Frankild and J{\o}rgensen, however, these papers focus on the derived category setting, whereas we are interested in the the abelian category setting. 

Concerning our work on relative CM modules in \secref{CM}: The duality \eqref{intro5} is new but related results, again in the derived category setting, can be found in \cite{AFrPJr04}, Porta, Shaul, and Yekutieli \cite[Sect.~7]{PSY}, and Vyas and Yekutieli \cite[Sect.~8]{VyasYekutieli} (MGM equivalence). 

\enlargethispage{4.2ex}

\section{Preliminaries and technical lemmas}

For an abelian category $\cA$, we write $\K{\cA}$ for its homotopy category.

\begin{ipg}  
  A chain map $\alpha \colon X \to Y$ between complexes $X$ and $Y$ in an abelian category is called a \emph{quasi-isomorphism} if $\H{n}{\alpha} \colon \H{n}{X} \to \H{n}{Y}$ is an isomorphism for every $n \in \mathbb{Z}$. 
  
For a complex $X$ and an integer $\ell$ we write $\upSigma^\ell X$ for the \emph{$\ell^\mathrm{\mspace{2mu}th}$ translate} of $X$; this complex is defined by $(\upSigma^\ell X)_n = X_{n-\ell}$ and \smash{$\partial^{\mspace{1mu}\upSigma^\ell X}_n = (-1)^\ell \partial^X_{n-\ell}$} for $n \in \mathbb{Z}$.
\end{ipg}

\begin{ipg}
  \label{P-I}
  If $\cA$ is an abelian category with enough projectives, then we write $\Fu{P}(A)$ for any projective re\-so\-lu\-tion of $A \in \cA$. By the unique, up to homotopy, lifting property of projective resolutions one gets a well-defined functor \mbox{$\Fu{P} \colon \cA \to \K{\cA}$}, and we write $\pi_A \colon \Fu{P}(A) \to A$ for the ca\-no\-ni\-cal quasi-isomorphism.
  
Dually, if $\cB$ is an abelian category with enough injectives, then we write $\Fu{I}(B)$ for any injective re\-so\-lu\-tion of $B \in \cB$. This yields a well-defined functor \mbox{$\Fu{I} \colon \cB \to \K{\cB}$} and we write $\iota_B \colon B \to \Fu{I}(B)$ for the ca\-no\-ni\-cal quasi-isomorphism.  
\end{ipg}
\begin{dfn}
  \label{dfn:concentrated}
   Let $\cA$ be an abelian category and let $\ell \in \mathbb{Z}$. A complex $X$ in $\cA$ is said to have its \emph{homology concentrated in degree $\ell$} if one has $\H{i}{X}=0$ for all $i \neq \ell$.
\end{dfn}

\begin{lem}
  \label{lem:P-iso}
  Let $\cA$ be an abelian category with enough projectives and let $\ell \in \mathbb{Z}$. Let $A$ be an object in $\cA$ and let $X$ be a complex in $\cA$ whose homology is concentrated in degree~$\ell$. 
There is an isomorphism of abelian groups, natural in both $A$ and $X$, given by:
\begin{displaymath}
    \xymatrix{
      \Hom{\cA}{A}{\H{\ell}{X}}
      \ar[r]^-{\smash{\text{\raisebox{0.5ex}{$u^\ell_{A,X}$}}}}_-{\cong}         
      &
      \Hom{\K{\cA}}{\Fu{P}(A)}{\upSigma^{-\ell}X}
    }\!,
\end{displaymath}        
whose inverse is induced by the functor $\H{0}{-}$. Furthermore, a morphism \mbox{$\sigma \colon A \to \H{\ell}{X}$} in $\cA$ is an isomorphism if and only if \smash{$u^\ell_{A,X}(\sigma) \colon \Fu{P}(A) \to \upSigma^{-\ell}X$} is a quasi-isomorphism.
\end{lem}

\begin{proof}
  Let $\D{\cA}$ be the derived category of $\cA$. As $\cA$ is a full subcatgory of $\D{\cA}$, we have $\Hom{\cA}{A}{\H{\ell}{X}} \cong \Hom{\D{\cA}}{A}{\H{\ell}{X}}$. In $\D{\cA}$ one has natural isomorphisms $A \cong \Fu{P}(A)$ and \smash{$\H{\ell}{X} \cong \upSigma^{-\ell}X$}, as the homology of $X$ is concentrated in degree $\ell$, and consequently $\Hom{\D{\cA}}{A}{\H{\ell}{X}} \cong \Hom{\D{\cA}}{\Fu{P}(A)}{\upSigma^{-\ell}X}$. It is well-known that $\Hom{\D{\cA}}{\Fu{P}(A)}{Y} \cong \Hom{\K{\cA}}{\Fu{P}(A)}{Y}$ for any complex $Y$ in $\cA$ since $\Fu{P}(A)$ is a bounded below complex of projectives. By composing these natural isomorphisms, the assertion follows.
\end{proof}

The next lemma is proved similarly.

\begin{lem}
  \label{lem:I-iso}
  Let $\cB$ be an abelian category with enough injectives and let $\ell \in \mathbb{Z}$. Let $B$ be an object in $\cB$ and let $Y$ be a complex in $\cB$ whose homology is concentrated in degree~$\ell$. 
There is an isomorphism of abelian groups, natural in both $B$ and $Y$, given by:  
\begin{displaymath}
    \xymatrix{
      \Hom{\cB}{\H{\ell}{Y}}{B}
      \ar[r]^-{\smash{\text{\raisebox{0.5ex}{$v^\ell_{Y,B}$}}}}_-{\cong}       
      &
      \Hom{\K{\cB}}{\upSigma^{-\ell}Y}{\Fu{I}(B)}
    }\!,
\end{displaymath}        
whose inverse is induced by the functor $\H{0}{-}$. Furthermore, a morphism \mbox{$\tau \colon \H{\ell}{Y} \to B$} in $\cB$ is an isomorphism if and only if \smash{$v^\ell_{Y,B}(\tau) \colon \upSigma^{-\ell}Y \to \Fu{I}(B)$} is a quasi-isomorphism.\qed
\end{lem}

\begin{ipg}
  \label{derived-functors}
  As in Mac~Lane~\cite[I\S2]{Mac}, a functor means a \textsl{covariant} functor. Let $\Fu{T} \colon \cA \to \cB$ be an additive (covariant) functor between abelian categories. Recall that if $\cA$ has enough projectives, then the \emph{$i^\mathrm{th}$ left derived functor} of $\Fu{T}$ is given by $\Ld{i}{\Fu{T}}(A) = \Hnop{i}\Fu{T}(P)$ where $P$ is any projective resolution of $A \in \cA$. If $\Fu{T}$ is right exact, then $\Ld{0}{\Fu{T}} = \Fu{T}$. Dually, if $\cA$~has~enough injectives, then the \emph{$i^\mathrm{th}$ right derived functor} of $\Fu{T}$ is given by $\Rd{i}{\Fu{T}}(A) = \Hnop{-i}\Fu{T}(I)$ where $I$ is any injective resolution of $A \in \cA$. And if $\Fu{T}$ is left exact, then $\Rd{0}{\Fu{T}} = \Fu{T}$.
  
  Consider now the opposite functor  $\Fu{T}^\mathrm{op} \colon \cA^\mathrm{op} \to \cB^\mathrm{op}$ of $\Fu{T}$. The category $\cA^\mathrm{op}$ has enough projectives (resp.~injectives) if and only if $\cA$ has enough injectives (resp.~projectives), and in this case one has $\Ld{i}{(\Fu{T}^\mathrm{op})} = (\Rd{i}{\Fu{T}})^\mathrm{op}$ (resp.~$\Rd{i}{(\Fu{T}^\mathrm{op})} = (\Ld{i}{\Fu{T}})^\mathrm{op}$).
\end{ipg}

If \mbox{$\Fu{S} \colon\mspace{-2mu} \cA \rightleftarrows \cB \colon \Fu{T}$} is an adjunction, where $\Fu{S}$ is the left adjoint of $\Fu{T}$, with unit $\eta \colon \Fu{Id}_{\cA} \to \Fu{T}\Fu{S}$ and counit $\varepsilon \colon \Fu{S}\Fu{T} \to \Fu{Id}_{\cB}$, then the composites
\smash{$\!\!\xymatrix@C=1.5pc{
   \Fu{S} 
   \ar[r]^-{\smash{\text{\raisebox{-0.2ex}{$\Fu{S}\eta$}}}} 
   & 
   \Fu{S}\Fu{T}\Fu{S}
   \ar[r]^-{\smash{\text{\raisebox{-0.3ex}{$\varepsilon\Fu{S}$}}}}  
   &
   \Fu{S}  
   }\!\!$}
and    
\smash{$\!\!\xymatrix@C=1.5pc{
   \Fu{T} 
   \ar[r]^-{\smash{\text{\raisebox{-0.2ex}{$\eta\Fu{T}$}}}} 
   & 
   \Fu{T}\Fu{S}\Fu{T}
   \ar[r]^-{\smash{\text{\raisebox{-0.3ex}{$\Fu{T}\varepsilon$}}}}  
   &
   \Fu{T}  
   }\!\!$}
are the identities on $\Fu{S}$ and $\Fu{T}$; see e.g.~\cite[IV\S1 Thm.~1]{Mac}. In the proof of   \thmref{equivalence} we will need the following sligthly more careful version of this fact.

\begin{lem}
  \label{lem:ST}
  Let \mbox{$\Fu{S} \colon\mspace{-2mu} \cA \rightleftarrows \cB \colon \Fu{T}$} be functors (not assumed to be an adjunction), let $\cA_0$ and $\cB_0$ be a full subcategories of $\cA$ and $\cB$, and assume that there is a natural bijection
\begin{displaymath}
  \xymatrix{
    \Hom{\cB}{\Fu{S}A}{B} \ar[r]^-{\smash{\text{\raisebox{0.2ex}{$k_{A,B}$}}}} &
    \Hom{\cA}{A}{\Fu{T}B}
  }
\end{displaymath}    
for $A \in \cA_0$ and $B \in \cB_0$. (We do not assume $\Fu{S}(\cA_0) \subseteq \cB_0$ and $\Fu{T}(\cB_0) \subseteq \cA_0$, so it is not given the functors $\Fu{S}$ and $\Fu{T}$ restrict to an adjunction $\cA_0 \rightleftarrows \cB_0$.) 

For every $A \in \cA_0$ which satisfies $\Fu{S}A \in \cB_0$ set $\eta_A = k_{A,\Fu{S}A}(1_{\Fu{S}A}) \colon A \to \Fu{T}\Fu{S}A$, and for every $B \in \cB_0$ which satisfies $\Fu{T}B \in \cA_0$ set $\varepsilon_B = k^{-1}_{\,\Fu{T}B,B}(1_{\Fu{T}B}) \colon \Fu{S}\Fu{T}B \to B$. The following hold:
\begin{prt}
\item If $A \in \cA$ is an object with $A, \Fu{T}\Fu{S}A \in \cA_0$ and $\Fu{S}A \in \cB_0$, then  \smash{$\!\!\xymatrix@C=1.7pc{
   \Fu{S}A 
   \ar[r]^-{\smash{\text{\raisebox{-0.1ex}{$\Fu{S}(\eta_A)$}}}} 
   & 
   \Fu{S}\Fu{T}\Fu{S}A
   \ar[r]^-{\smash{\text{\raisebox{-0.1ex}{$\varepsilon_{\Fu{S}A}$}}}}  
   &
   \Fu{S}A  
   }\!\!$} is the identity on $\,\Fu{S}A$.
   
\item If $B \in \cB$ is an object with $B, \Fu{S}\Fu{T}B \in \cB_0$ and $\Fu{T}B \in \cA_0$, then  
\smash{$\!\!\xymatrix@C=1.7pc{
   \Fu{T}B 
   \ar[r]^-{\smash{\text{\raisebox{-0.1ex}{$\eta_{\Fu{T}B}$}}}} 
   & 
   \Fu{T}\Fu{S}\Fu{T}B
   \ar[r]^-{\smash{\text{\raisebox{-0.1ex}{$\Fu{T}(\varepsilon_B)$}}}}  
   &
   \Fu{T}B  
   }\!\!$} is the identity on $\,\Fu{T}B$.
\end{prt}
\end{lem}

\begin{proof}
  Inspect the proof of \cite[IV\S1 Thm.~1]{Mac}.
\end{proof}

\enlargethispage{2ex}

\section{Fixed and cofixed objects}
\label{sec:Fixed}

In this section, we prove our main result, \thmref{equivalence}, which in certain situations takes the simpler form of \thmref{equivalence-simplified}.

\begin{stp}
   \label{stp:setup}
   Throughout, $\cA$ is an abelian category with enough projectives and $\cB$ is an abe\-lian category with enough injectives. Furthermore, \mbox{$\Fu{F} \colon\mspace{-2mu} \cA \rightleftarrows \cB \colon \Fu{G}$} is an adjunction with $\Fu{F}$ being left adjoint of $\Fu{G}$. We write $h_{A,B} \colon \Hom{\cB}{\Fu{F}A}{B} \to \Hom{\cA}{A}{\Fu{G}B}$ for the given na\-tu\-ral iso\-mor\-phism and denote by $\eta_A \colon A \to \Fu{G}\Fu{F}A$ and  $\varepsilon_B \colon \Fu{F}\Fu{G}B \to B$ the unit and counit.
\end{stp}

The following examples of \stpref{setup} are useful to have in mind.

\begin{exa}
  \label{exa:2}
  Let $\Gamma$ and $\Lambda$ be rings and let $T = {}_\Gamma T_{\!\Lambda}$ be a $(\Gamma,\Lambda)$-bimodule. The functors
\begin{equation*} 
  \xymatrix@C=5.5pc{
    \Mod{\Lambda}
    \ar@<0.6ex>[r]^-{\Fu{F}\mspace{3mu}=\mspace{3mu}T \otimes_\Lambda-} 
    & 
    \Mod{\Gamma}
    \ar@<0.6ex>[l]^-{\Fu{G}\mspace{3mu}=\mspace{3mu}\Hom{\Gamma}{T}{-}}    
  }
\end{equation*} 
constitute an adjunction with unit and counit:
\begin{eqnarray*}
  \eta_A \colon A \longrightarrow \Hom{\Gamma}{T}{T \otimes_\Lambda A}
  &\text{ given by }& \eta_A(a)(t) = t \otimes a \quad \text{and}
  \\
  \varepsilon_B \colon T \otimes_\Lambda\Hom{\Gamma}{T}{B} \longrightarrow B
  &\text{ given by }& \varepsilon_B(t \otimes \beta) = \beta(t)\,.
\end{eqnarray*}
If $\Gamma$ and $\Lambda$ are artin algebras and the modules ${}_\Gamma T$ and $T_{\!\Lambda}$ are finitely generated, then the above restricts to an adjunction between the subcategories of finitely generated modules:
\begin{equation*} 
  \xymatrix@C=5.5pc{
    \mod{\Lambda}
    \ar@<0.6ex>[r]^-{\Fu{F}\mspace{3mu}=\mspace{3mu}T \otimes_\Lambda-} 
    & 
    \mod{\Gamma}
    \ar@<0.6ex>[l]^-{\Fu{G}\mspace{3mu}=\mspace{3mu}\Hom{\Gamma}{T}{-}}    
  }\!.
\end{equation*} 
In this case the category $\mod{\Lambda}$ has enough projectives and $\mod{\Gamma}$ has enough injectives, see e.g. \cite[II.3 Cor.~3.4]{rta}, so the situation satisfies \stpref{setup}.

Finally, we note that $\Ld{i}{\Fu{F}} = \Tor{\Lambda}{i}{T}{-}$ and $\Rd{i}{\Fu{G}} = \Ext{\Gamma}{i}{T}{-}$.
\end{exa}

\enlargethispage{1.3ex}

For a ring $\Lambda$ we write $\Lo$ for the opposite ring.

\begin{exa}
  \label{exa:1}
  Let $\Gamma$ and $\Lambda$ be rings and let $T = {}_\Gamma T_{\!\Lambda}$ be a $(\Gamma,\Lambda)$-bimodule. The functors
\begin{displaymath} 
  \xymatrix@C=6pc{
    \Mod{\Gamma}
    \ar@<0.6ex>[r]^-{\Fu{F}\mspace{3mu}=\mspace{3mu}\Hom{\Gamma}{-}{T}^\mathrm{op}} 
    & 
    \Mod{\Lo}^\mathrm{op} 
    \ar@<0.6ex>[l]^-{\Fu{G}\mspace{3mu}=\mspace{3mu}\Hom{\Lo}{-}{T}}    
  }
\end{displaymath}
constitute an adjunction whose unit and counit are the so-called biduality homomorphisms:
\begin{eqnarray*}
  \eta_A \colon A \longrightarrow \Hom{\Lo}{\Hom{\Gamma}{A}{T}}{T}
  &\text{ given by }& \eta_A(a)(\alpha) = \alpha(a) \quad \text{and}
  \\
  \varepsilon_B \colon B \longrightarrow \Hom{\Gamma}{\Hom{\Lo}{B}{T}}{T}
  &\text{ given by }& \varepsilon_B(b)(\beta) = \beta(b)\,.
\end{eqnarray*}
(Note that \emph{a priori} the counit is a mor\-phism $\Fu{F}\Fu{G}B \to B$ in $\Mod{\Lo}^\mathrm{op}$, but that corresponds to the morphism $B \to \Fu{F}\Fu{G}B$ in $\Mod{\Lo}$ displayed above.)

If $\Gamma$ is left coherent and $\Lambda$ is right coherent, then the categories $\mod{\Gamma}$ and $\mod{\Lo}$ of finitely presented $\Gamma$- and $\Lo$-modules are abelian with enough projectives (and hence the category $\mod{\Lo}^\mathrm{op}$ is abelian with enough injectives). In this case, and if the modules ${}_\Gamma T$ and $T_{\!\Lambda}$ are finitely presented, the above restricts to an adjunction:
\begin{displaymath} 
  \xymatrix@C=6pc{
    \mod{\Gamma}
    \ar@<0.6ex>[r]^-{\Fu{F}\mspace{3mu}=\mspace{3mu}\Hom{\Gamma}{-}{T}^\mathrm{op}} 
    & 
    \mod{\Lo}^\mathrm{op}
    \ar@<0.6ex>[l]^-{\Fu{G}\mspace{3mu}=\mspace{3mu}\Hom{\Lo}{-}{T}}    
  }\!.
\end{displaymath}

Finally, we note that $\Ld{i}{\Fu{F}} = \Ext{\Gamma}{i}{-}{T}^\mathrm{op}$ and $\Rd{i}{\Fu{G}} = \Ext{\Lo}{i}{-}{T}$ by \ref{derived-functors}.
\end{exa}

\begin{prp}
  \label{prp:key}
  Let $\ell$ be an integer. For $A \in \cA$ that satisfies $\Ld{i}{\Fu{F}}(A)=0$ for all $i \neq \ell$, and for $B \in \cB$ that satisfies $\Rd{i}{\Fu{G}}(B)=0$ for all $i \neq \ell$, there is a natural isomorphism:
\begin{displaymath}
  \xymatrix{
    \Hom{\cB}{\Ld{\ell}{\Fu{F}}(A)}{B}
    \ar[r]^-{\smash{\text{\raisebox{0.5ex}{$h^\ell_{A,B}$}}}}_-{\cong}         
    &
    \Hom{\cA}{A}{\Rd{\ell}{\Fu{G}}(B)}
  }\!.
\end{displaymath}  
\end{prp}

\begin{proof}
  The assumptions mean that the homology of the complex $\Fu{F}(\Fu{P}(A))$ is concentrated in degree $\ell$ and that the homology of $\Fu{G}(\Fu{I}(B))$ is concentrated in degree $-\ell$. We now define $h^\ell_{A,B}$ to be the unique homomorphism (which is forced to be an isomorphism) that makes the following diagram commutative:
\begin{equation}
  \label{eq:big}
  \begin{gathered}
  \xymatrix@R=1.5pc@C=4pc{
    \Hom{\cB}{\Ld{\ell}{\Fu{F}}(A)}{B} \ar@{..>}[r]^-{h^\ell_{A,B}}
    \ar@{=}[d]
    &
    \Hom{\cA}{A}{\Rd{\ell}{\Fu{G}}(B)}
    \ar@{=}[d]
    \\
    \Hom{\cB}{\Hnop{\ell}{\Fu{F}(\Fu{P}(A))}}{B}
    \ar[d]^-{\cong}_-{v^\ell_{\Fu{F}(\Fu{P}(A)),B}}
    & 
    \Hom{\cA}{A}{\Hnop{-\ell}{\Fu{G}(\Fu{I}(B))}}
    \ar[d]_-{\cong}^-{u^{-\ell}_{A,\Fu{G}(\Fu{I}(B))}}    
    \\
    \Hom{\K{\cB}}{\upSigma^{-\ell}\Fu{F}(\Fu{P}(A))}{\Fu{I}(B)}
    \ar[d]^-{\cong}_-{\upSigma^{\ell}(-)}        
    &
    \Hom{\K{\cA}}{\Fu{P}(A)}{\upSigma^{\ell}\Fu{G}(\Fu{I}(B))}
    \ar@{=}[d]    
    \\
    \Hom{\K{\cB}}{\Fu{F}(\Fu{P}(A))}{\upSigma^{\ell}\Fu{I}(B)}
    \ar[r]^-{\mathrm{adjunction}}_-{\cong}       
    &
    \Hom{\K{\cA}}{\Fu{P}(A)}{\Fu{G}(\upSigma^{\ell}\Fu{I}(B))}\,.
    }
  \end{gathered}    
\end{equation}    
The vertical isomorphisms come from \lemref[Lemmas~]{P-iso} and \lemref[]{I-iso}. The adjunction \mbox{$\Fu{F} \colon\mspace{-2mu} \cA \rightleftarrows \cB \colon \Fu{G}$} induces an adjunction $\K{\cA} \rightleftarrows \K{\cB}$ by degreewise application of the functors $\Fu{F}$ and $\Fu{G}$; this explains the lower vertical isomorphism in the diagram. Finally, we note that all the displayed isomorphisms are natural~in~$A$~and~$B$.
\end{proof}

\begin{dfn}
  \label{dfn:unit-counit}
  Let $\ell$ be an integer. If $A \in \cA$ satisfies $\Ld{i}{\Fu{F}}(A)=0=\Rd{i}{\Fu{G}}(\Ld{\ell}{\Fu{F}}(A))$ for all $i \neq \ell$, then we can apply \prpref{key} to $B=\Ld{\ell}{\Fu{F}}(A)$, and thereby obtain a morphism:
\begin{displaymath}
  \eta^\ell_A \colon A \longrightarrow \Rd{\ell}{\Fu{G}}(\Ld{\ell}{\Fu{F}}(A))
  \quad \text{defined by} \quad  \eta^\ell_A=h^\ell_{A,\Ld{\ell}{\Fu{F}}(A)}(1_{\Ld{\ell}{\Fu{F}}(A)})\;.
\end{displaymath}  
Similarly, if $B \in \cB$ has $\Rd{i}{\Fu{G}}(B)=0=\Ld{i}{\Fu{F}}(\Rd{\ell}{\Fu{G}}(B))$ for all $i \neq \ell$, then we get a morphism
\begin{displaymath}
  \varepsilon^\ell_B \colon \Ld{\ell}{\Fu{F}}(\Rd{\ell}{\Fu{G}}(B)) \longrightarrow B
  \quad \text{defined by} \quad  \varepsilon^\ell_B=(h^\ell_{\Rd{\ell}{\Fu{G}}(B),B})^{-1}(1_{\Rd{\ell}{\Fu{G}}(B)})\;.
\end{displaymath}   
\end{dfn}

\begin{rmk}
The proofs of \lemref[Lemmas~]{P-iso} and \lemref[]{I-iso} show how the maps \smash{$u^\ell_{A,X}$} and \smash{$v^\ell_{Y,B}$} act, and the diagram \eqref{big} shows how \smash{$h^\ell_{A,B}$} is a composition of these maps and the given adjunction. This tells us how \smash{$h^\ell_{A,B}$} acts. It can verified that for $\ell=0$ the isomorphism \smash{$h^\ell_{A,B}=h^0_{A,B}$} co\-in\-cides with the given natural isomorphism $h_{A,B}$ from \stpref{setup}, and hence \smash{$\eta^0_A$} and \smash{$\varepsilon^0_B$} from \dfnref{unit-counit} coincide with the unit $\eta_A$ and the counit $\varepsilon_B$ of the adjunction $(\Fu{F},\Fu{G})$.
\end{rmk}

The following is the key definition in this paper.

\begin{dfn}
 \label{dfn:Fix-Cofix}
 Let $\ell$ be an integer. An object $A \in \cA$ is called \emph{$\ell$-fixed} with respect to the adjunc\-tion $(\Fu{F},\Fu{G})$ if it satisfies the following three conditions:
\begin{prt}
  \item[\textnormal{(i)}] $\Ld{i}{\Fu{F}}(A)=0$ for all $i \neq \ell$.
  \item[\textnormal{(ii)}] $\Rd{i}{\Fu{G}}(\Ld{\ell}{\Fu{F}}(A))=0$ for all $i \neq \ell$.
  \item[\textnormal{(iii)}] The morphism $\eta^\ell_A \colon A \to \Rd{\ell}{\Fu{G}}(\Ld{\ell}{\Fu{F}}(A))$ is an isomorphism.
\end{prt}
The full subcategory of $\cA$ whose objects are the $\ell$-fixed ones is denoted by $\Fix[]{\ell}{\cA}$.

Dually, an object $B \in \cB$ is \emph{$\ell$-cofixed} with respect to $(\Fu{F},\Fu{G})$ if it satisfies:
\begin{prt}
  \item[\textnormal{(i$'$)}] $\Rd{i}{\Fu{G}}(B)=0$ for all $i \neq \ell$.
  \item[\textnormal{(ii$'$)}] $\Ld{i}{\Fu{F}}(\Rd{\ell}{\Fu{G}}(B))=0$ for all $i \neq \ell$.
  \item[\textnormal{(iii$'$)}] The morphism $\varepsilon^\ell_B \colon \Ld{\ell}{\Fu{F}}(\Rd{\ell}{\Fu{G}}(B)) \to B$ is an isomorphism.
\end{prt}
The full subcategory of $\cB$ whose objects are the $\ell$-cofixed ones is denoted by $\coFix[]{\ell}{\cB}$.
\end{dfn}

The categories of $\ell$-fixed objects in $\cA$ and $\ell$-cofixed objects in $\cB$ are, in fact, equivalent:

\begin{thm}
  \label{thm:equivalence}
  In the notation from \stpref{setup} and \dfnref{Fix-Cofix} there is for every integer~$\ell$ an adjoint equivalence of categories:
\begin{displaymath}
  \xymatrix@C=3pc{
    \Fix[]{\ell}{\cA}\, \ar@<0.6ex>[r]^-{\Ld{\ell}{\Fu{F}}} & \,\coFix[]{\ell}{\cB} \ar@<0.6ex>[l]^-{\Rd{\ell}{\Fu{G}}}
  }\!.
\end{displaymath}  
\end{thm}

\begin{proof}
Let $\cA_0$, respectively, $\cB_0$, be the full subcategory of $\cA$, respectively, $\cB$, whose~\mbox{objects} satisfy condition (i), respectively, (i$'$), in \dfnref{Fix-Cofix}. 
By \prpref{key} we may apply \lemref{ST} to these choices of $\cA_0$ and $\cB_0$ and to $\Fu{S}=\Ld{\ell}{\Fu{F}}$ and $\Fu{T}=\Rd{\ell}{\Fu{G}}$. From part (a) of that lemma (and from \dfnref{unit-counit}) we conclude that if $A \in \cA$ satisfies the conditions
\begin{prt}
\setlength{\itemsep}{0.25ex}
\item[($1^\circ$)] $A \in \cA_0$, that is, $A$ satisfies \dfnref[]{Fix-Cofix}(i),
\item[($2^\circ$)] $\Fu{S}A \in \cB_0$, that is, $A$ satisfies \dfnref[]{Fix-Cofix}(ii), and 
\item[($3^\circ$)] $\Fu{T}\Fu{S}A \in \cA_0$, that is, $B=\Ld{\ell}{\Fu{F}}(A)$ satisfies \dfnref[]{Fix-Cofix}(ii$'$),
\end{prt}
then one has \smash{$\varepsilon^\ell_{\Ld{\ell}{\Fu{F}}(A)} \circ \Ld{\ell}{\Fu{F}}(\eta^\ell_A) = 1_{\Ld{\ell}{\Fu{F}}(A)}$}. We now see that the functor $\Ld{\ell}{\Fu{F}}$ maps $\Fix[]{\ell}{\cA}$~to $\coFix[]{\ell}{\cB}$, indeed, if $A$ belongs to $\Fix[]{\ell}{\cA}$, then $B:=\Ld{\ell}{\Fu{F}}(A)$ satisfies (i$'$) as $A$ satisfies (ii), and $B$ satisfies (ii$'$) since $A$ satisfies (iii) and (i). In particular, conditions \mbox{($1^\circ$)--($3^\circ$)} above hold, and hence \smash{$\varepsilon^\ell_{B} \circ \Ld{\ell}{\Fu{F}}(\eta^\ell_A) = 1_{\!B}$}. Since $\eta^\ell_A$ is an isomorphism by (iii), it follows that $\varepsilon^\ell_{B}$ is an isomorphism as well, that is, $B$ satisfies condition (iii$'$).

Similar arguments show that the functor $\Rd{\ell}{\Fu{G}}$ maps $\coFix[]{\ell}{\cB}$ to $\Fix[]{\ell}{\cA}$. Now \prpref{key} and \dfnref{unit-counit} show that $(\Ld{\ell}{\Fu{F}},\Rd{\ell}{\Fu{G}})$ gives an adjunction between the categories $\Fix[]{\ell}{\cA}$ to $\coFix[]{\ell}{\cB}$ with unit $\eta^\ell$ and counit~$\varepsilon^\ell$. Finally, conditions \dfnref[]{Fix-Cofix}(iii) and (iii$'$) show that $(\Ld{\ell}{\Fu{F}},\Rd{\ell}{\Fu{G}})$ yields an adjoint equivalence between $\Fix[]{\ell}{\cA}$ and $\coFix[]{\ell}{\cB}$.
\end{proof}

\begin{lem}
  \label{lem:extension}
  The categories $\Fix[]{\ell}{\cA}$ and $\coFix[]{\ell}{\cB}$ are closed under direct summands and extensions in $\cA$ and $\cB$, respectively.
\end{lem}

\begin{proof}
  Straightforward from the definitions.
\end{proof}

The next lemma (which does not use that $\Fu{G}$ is a right adjoint, but only that it is left~exact) is variant of Hartshorne \cite[III\S1 Prop.~1.2A]{Hartshorne}. Recall that $B \in \cB$ is called \emph{$\Fu{G}$-acyclic} if $\Rd{i}{\Fu{G}}(B)=0$ for all $i>0$. Similarly, $A \in \cA$ is called \emph{$\Fu{F}$-acyclic} if $\Ld{i}{\Fu{F}}(A)=0$ for all $i>0$.

Also recall that an additive functor $\Fu{T}$ between abelian categories is said to have \emph{finite homological dimension}, respectively, \emph{finite cohomological dimension}, if one has $\Ld{d}{\Fu{T}}=0$, respectively, $\Rd{d}{\Fu{T}}=0$, for some integer $d \geqslant 0$.

\begin{lem}
  \label{lem:G-on-quiso}
  Let $\gamma \colon X \to Y$ be a quasi-isomorphism between complexes in $\cB$ that consist of $\,\Fu{G}$-acyclic objects. If $\Fu{G}$ has finite cohomological dimension, then $\Fu{G}\gamma \colon \Fu{G}X \to \Fu{G}Y$ is a quasi-isomorphism.
\end{lem}

\begin{proof}
  This is left as an exercise to the reader.
\end{proof}

Under suitable assumptions we obtain in \prpref[Propositions~]{Fix-simplified} and \prpref[]{coFix-simplified} below simplified descriptions of the categories $\Fix[]{\ell}{\cA}$ and $\coFix[]{\ell}{\cB}$.

\begin{dfn}
  \label{dfn:tilting-adjunction}
  The adjunction $(\Fu{F},\Fu{G})$ from \stpref{setup} is called a \emph{tilting adjunction} if it satisfies the following four conditions:
  \begin{prt}
    \item[\con{TA1\hspace*{-0.5pt}}\hspace*{-2pt}] For every projective object $P \in \cA$ the object $\Fu{F}(P)$ is $\Fu{G}$-acyclic and the unit~of~adjunc\-tion $\eta_P \colon P \to \Fu{G}\Fu{F}(P)$ is an isomorphism. In other words,  $\Proj{\cA} \subseteq \Fix[]{0}{\cA}$.
    
    \item[\con{TA2\hspace*{-0.5pt}}\hspace*{-2pt}] The functor $\Fu{G}$ has finite cohomological dimension.
  
\item[\con{TA3\hspace*{-0.5pt}}\hspace*{-2pt}] For every injective object $I \in \cB$ the object $\Fu{G}(I)$ is $\Fu{F}$-acyclic and the counit~of~adjunc\-tion $\varepsilon_I \colon \Fu{F}\Fu{G}(I) \to I$ is an isomorphism. In other words,  $\Inj{\cB} \subseteq \coFix[]{0}{\cB}$.
    
    \item[\con{TA4\hspace*{-0.5pt}}\hspace*{-2pt}] The functor $\Fu{F}$ has finite homological dimension. 
  \end{prt}
\end{dfn}

\begin{prp}
  \label{prp:Fix-simplified}
  If the adjunction $(\Fu{F},\Fu{G})$ satisfies conditions \con{TA1} and \con{TA2} in \dfnref{tilting-adjunction}, then for every integer $\ell$ and every $A \in \cA$ one has:
  \begin{displaymath}
    A \in \Fix[]{\ell}{\cA} 
    \quad \iff \quad
    \textnormal{$\Ld{i}{\Fu{F}}(A)=0$ for all $i \neq \ell$\,.}
  \end{displaymath}  
  In other words, in this case, conditions \textnormal{(ii)} and \textnormal{(iii)} in \dfnref{Fix-Cofix} are automatic.
\end{prp}

\begin{proof}
  The implication ``$\Rightarrow$'' holds by \dfnref{Fix-Cofix}(i). Conversely, assume \mbox{$\Ld{i}{\Fu{F}}(A)=0$} for all $i \neq \ell$. We must argue that conditions (ii) and (iii) in \dfnref{Fix-Cofix} hold as well. Let $P$ be a projective resolution of $A$ and let $I$ be a injective resolution of $\Ld{\ell}{\Fu{F}}(A) = \Hnop{\ell}{\Fu{F}(P)}$. Our assumption means that the homology of the complex $\Fu{F}(P)$ is concentrated in degree $\ell$. With $B=\Ld{\ell}{\Fu{F}}(A)$ we now consider the following part of the diagram \eqref{big}:
\begin{equation}
  \label{eq:big2}
  \begin{gathered}
  \xymatrix@R=1.5pc@C=4pc{
    \Hom{\cB}{\Ld{\ell}{\Fu{F}}(A)}{\Ld{\ell}{\Fu{F}}(A)}
    \ar@{=}[d]
    &
    {}
    \\
    \Hom{\cB}{\Hnop{\ell}{\Fu{F}(P)}}{\Ld{\ell}{\Fu{F}}(A)}
    \ar[d]^-{\cong}_-{v \,\,:=\,\, v^\ell_{\Fu{F}(P),\,\Ld{\ell}{\Fu{F}}(A)}}
    & 
    {}
    \\
    \Hom{\K{\cB}}{\upSigma^{-\ell}\Fu{F}(P)}{I}
    \ar[d]^-{\cong}_-{\upSigma^{\ell}(-)}        
    &
    \Hom{\K{\cA}}{P}{\upSigma^{\ell}\Fu{G}(I)}
    \ar@{=}[d]    
    \\
    \Hom{\K{\cB}}{\Fu{F}(P)}{\upSigma^{\ell}I}
    \ar[r]^-{\mathrm{adjunction}}_-{\cong}       
    &
    \Hom{\K{\cA}}{P}{\Fu{G}(\upSigma^{\ell}I)}\,.
    }
  \end{gathered}    
\end{equation}      
Set \smash{$\gamma = v(1_{\Ld{\ell}{\Fu{F}}(A)}) \colon \upSigma^{-\ell}\Fu{F}(P) \to I$} in $\K{\cB}$, which is a quasi-isomor\-phism by \lemref{I-iso}. Under the maps in \eqref{big2}, the identity morphism \smash{$1_{\Ld{\ell}{\Fu{F}}(A)}$} is mapped to $\theta \in \Hom{\K{\cA}}{P}{\upSigma^{\ell}\Fu{G}(I)}$ given by $\theta = \Fu{G}(\upSigma^{\ell}\gamma) \circ \eta_P$, that is, $\theta$ is the composite:
\begin{equation}
  \label{eq:composite}
  \xymatrix{
    P \ar[r]^-{\eta_P} & \Fu{G}(\Fu{F}(P)) \ar[r]^-{\Fu{G}(\upSigma^{\ell}\gamma)} & \Fu{G}(\upSigma^{\ell}I) = \upSigma^{\ell}\Fu{G}(I)
  }\!.
\end{equation}
Here $\eta_P$ is an isomorphism by assumption \con{TA1}. As $\Fu{F}(P)$ and $\upSigma^{\ell}I$ consist of \mbox{$\Fu{G}$-acyclic} objects---again by \con{TA1}---the other assumption \con{TA2} together with \lemref{G-on-quiso} imply that the quasi-isomorphism~\mbox{$\upSigma^{\ell}\gamma \colon \Fu{F}(P) \to \upSigma^{\ell}I$} remains to be a quasi-isomorphism after application of $\Fu{G}$. Consequently, $\theta \colon P \to \upSigma^{\ell}\Fu{G}(I)$ is a quasi-isomorphism. As the homology of $P$ is concentrated in degree $0$ we get
\begin{displaymath}  
    \Rd{i}{\Fu{G}}(\Ld{\ell}{\Fu{F}}(A)) = \Hnop{-i}{\Fu{G}(I)} \cong
    \H{-i}{\upSigma^{-\ell}P} = \H{-i+\ell}{P} = 0 
  \quad \text{for all} \quad i \neq \ell\;,
\end{displaymath}
which proves condition \dfnref[]{Fix-Cofix}(ii). It now makes sense to consider the remaining part of~the~dia\-gram \eqref{big} (still with $B=\Ld{\ell}{\Fu{F}}(A)$), which gives us the middle equality below:
\begin{displaymath}  
   \eta^\ell_A = h^\ell_{A,\Ld{\ell}{\Fu{F}}(A)}(1_{\Ld{\ell}{\Fu{F}}(A)}) = 
   (u^{-\ell}_{A,\Fu{G}(I)})^{-1}(\theta) = \H{0}{\theta}\;.
\end{displaymath}  
Here the first equality is by \dfnref{unit-counit} and the last equality is by \lemref{P-iso}. As $\theta$ is a quasi-isomorphism, $\eta^\ell_A = \H{0}{\theta}$ is an isomorphism, and hence condition \dfnref[]{Fix-Cofix}(iii) holds.
\end{proof}

\begin{prp}
  \label{prp:coFix-simplified}
  If the adjunction $(\Fu{F},\Fu{G})$ satisfies conditions \con{TA3} and \con{TA4} in \dfnref{tilting-adjunction}, then for every integer $\ell$ and every $B \in \cB$ one has:
  \begin{displaymath}
    B \in \coFix[]{\ell}{\cB} 
    \quad \iff \quad
    \textnormal{$\Rd{i}{\Fu{G}}(B)=0$ for all $i \neq \ell$\,.}
  \end{displaymath}  
  In other words, in this case, conditions \textnormal{(ii$'$)} and \textnormal{(iii$'$)} in \dfnref{Fix-Cofix} are automatic.
\end{prp}

\begin{proof}
  Similar to the proof of \prpref{Fix-simplified}. 
\end{proof}

\begin{thm}
  \label{thm:equivalence-simplified}
  If $(\Fu{F},\Fu{G})$ is a tilting adjunction, then there is an adjoint equivalence:
\begin{displaymath}
  \xymatrix@C=3pc{
  {
  \{ 
     A \in \cA \ | \ 
     \Ld{i}{\Fu{F}}(A)=0 \text{ for all } i \neq \ell
  \mspace{1mu}\}
  }
  \ar@<0.6ex>[r]^-{\Ld{\ell}{\Fu{F}}}
  &
  {
  \{ 
     B \in \cB  \ | \
     \Rd{i}{\Fu{G}}(B)=0 \text{ for all } i \neq \ell
  \mspace{1mu}\}
  }
  \ar@<0.6ex>[l]^-{\Rd{\ell}{\Fu{G}}}
  }\!.
\end{displaymath}      
\end{thm}

\begin{proof}
  In view of \prpref[Propositions~]{Fix-simplified} and \prpref[]{coFix-simplified} this is immediate from \thmref{equivalence}.
\end{proof}

\section{Applications to tilting theory and commutative algebra}
\label{sec:Applications}

In this section, we demonstrate how some classic equivalences of categories from tilting theory and commutative algebra are special cases of \thmref[Theorems~]{equivalence} and \thmref[]{equivalence-simplified}. 

Tilting modules of projective dimension $\leqslant 1$ over artin algebras were originally conside\-red by Brenner and Butler \cite{BrennerButler} (although the term ``tilting'' first appeared in \cite{HappelRingel} by Happel and Ringel). Later people, such as Happel \cite[III\S3]{Happel} and Miyashita~\cite{Miyashita}, studied tilting modules of arbitrary finite projective dimension over general rings. 
If $\Gamma$ is an artin algebra with canonical duality $\operatorname{D} \colon \mod{\Gamma} \to \mod{\Go}$, then a finitely generated $\Gamma$-module $C$ is called cotilting if the $\Go$-module $\operatorname{D}(C)$ is tilting. 

The so-called Wakamatsu tilting modules constitute a good common generalization of both tilting and cotilting modules. In \cite{Wakamatsu88} Wakamatsu  introduced such modules over artin algebras; the following more general definition can be found in Wakamatsu \cite[Sec.~3]{Wakamatsu04}.

\begin{dfn}[Wakamatsu]
  \label{dfn:Wakamatsu-tilting}
  Let $\Gamma$ and $\Lambda$ be rings. A \emph{Wakamatsu tilting module} for the pair $(\Gamma,\Lambda)$ is a $(\Gamma,\Lambda)$-bimodule $T = {}_\Gamma T_{\!\Lambda}$ that satisfies the following conditions:
\begin{prt}
  \item[\con{W1}] The modules ${}_\Gamma T$ and $T_{\!\Lambda}$ admit resolutions by finitely generated projective modules.
  \item[\con{W2}] $\Ext{\Gamma}{i}{T}{T}=0$ and $\Ext{\Lo}{i}{T}{T}=0$ for all $i>0$.
  \item[\con{W3}] The canonical map $\Lambda \to \Hom{\Gamma}{T}{T}$ is an isomorphism of $(\Lambda,\Lambda)$-bimodules and the canonical map $\Gamma \to \Hom{\Lo}{T}{T}$ is an isomorphism of $(\Gamma,\Gamma)$-bimodules.
\end{prt}
\end{dfn}

The original version of the next result is a classic theorem by Brenner and Butler \cite{BrennerButler}; it was later improved by Happel \cite[III\S3]{Happel} and Miyashita~\cite[Thm.~1.16]{Miyashita}. All of these results are covered by following corollary of \thmref{equivalence-simplified}.

\begin{cor}[Brenner--Butler and Happel]
  \label{cor:BBH}
  Let $\Gamma$ and $\Lambda$ be rings. If $\,T = {}_\Gamma T_{\!\Lambda}$ is a Wakamatsu tilting module for which $\pd{\Gamma}{T}$ and $\pd{\Lo}{T}$ are finite, then there is for every $\ell \in \mathbb{Z}$ an adjoint equivalence:
\begin{displaymath}
  \xymatrix@C=5pc{
  {
  \left\{ \!\textnormal{\normalsize $M \in \Mod{\Lambda}$} \,\left|\!\!\!
    \begin{array}{c}
      \textnormal{\normalsize $\Tor{\Lambda}{i}{T}{M}=0$} \\
      \textnormal{\normalsize for all $i \neq \ell$} 
    \end{array}
  \right. \!\!\!\!\!\right\}
  }
  \ar@<0.8ex>[r]^-{\Tor{\Lambda}{\ell}{T}{-}}
  &
  {
  \left\{ \!\textnormal{\normalsize $N \in \Mod{\Gamma}$} \,\left|\!\!\!
    \begin{array}{c}
      \textnormal{\normalsize $\Ext{\Gamma}{i}{T}{N}=0$} \\
      \textnormal{\normalsize for all $i \neq \ell$} 
    \end{array}
  \right. \!\!\!\!\!\right\}
  }  \ar@<0.8ex>[l]^-{\Ext{\Gamma}{\ell}{T}{-}}
  }\!.
\end{displaymath}      
If $\Gamma$ and $\Lambda$ are artian algebras and the modules ${}_\Gamma T$ and $T_{\!\Lambda}$ are finitely generated, then the categories $\Mod{\Lambda}$ and $\Mod{\Gamma}$ may be replaced by $\mod{\Lambda}$ and $\mod{\Gamma}$.
\end{cor}

\begin{proof}
  Consider the adjunction 
\smash{\mbox{$\!\!\xymatrix@C=1.2pc{
    T \otimes_\Lambda- \colon \Mod{\Lambda}
    \ar@<0.5ex>[r]
    & 
    \Mod{\Gamma} : \Hom{\Gamma}{T}{-}
    \ar@<0.5ex>[l]
  }\!\!$}}
from \exaref{2}. Under the given assumptions on $T$, it is straightforward to verify that this is a tilting adjunction in the sense of \dfnref{tilting-adjunction}. Now apply \thmref{equivalence-simplified}.
\end{proof}

The following corollary of \thmref{equivalence-simplified} recovers
\cite[Prop.~8.1]{Wakamatsu04} by Wakamatsu.

\begin{cor}[Wakamatsu]
  \label{cor:Wakamatsu}
    Assume that $\Gamma$ is a left coherent ring and that $\Lambda$ is right coherent ring. If $T = {}_\Gamma T_{\!\Lambda}$ is a Wakamatsu tilting module for which $\id{\Gamma}{T}$ and $\id{\Lo}{T}$ are finite, then there is for every $\ell \in \mathbb{Z}$ an adjoint equivalence:
\begin{displaymath}
  \xymatrix@C=5pc{
  {
  \left\{ \!\textnormal{\normalsize $M \in \mod{\Gamma}$} \,\left|\!\!\!
    \begin{array}{c}
      \textnormal{\normalsize $\Ext{\Gamma}{i}{M}{T}=0$} \\
      \textnormal{\normalsize for all $i \neq \ell$} 
    \end{array}
  \right. \!\!\!\!\!\right\}
  }
  \ar@<0.8ex>[r]^-{\Ext{\Gamma}{\ell}{-}{T}^{\mathrm{op}}
  }
  &
  {
  \left\{ \!\textnormal{\normalsize $N \in \mod{\Lo}$} \,\left|\!\!\!
    \begin{array}{c}
      \textnormal{\normalsize $\Ext{\Lo}{i}{N}{T}=0$} \\
      \textnormal{\normalsize for all $i \neq \ell$} 
    \end{array}
  \right. \!\!\!\!\!\right\}^{\!\!\mathrm{op}}
  }  \ar@<0.8ex>[l]^-{\Ext{\Lo}{\ell}{-}{T}}
  }\!.
\end{displaymath}          
\end{cor}

\begin{proof}
  Consider the adjunction 
\smash{\mbox{$\!\!\xymatrix@C=1.2pc{
    \Hom{\Gamma}{-}{T}^\mathrm{op} \colon \mod{\Gamma}
    \ar@<0.5ex>[r]
    & 
    \mod{\Lo}^\mathrm{op} : \Hom{\Lo}{-}{T}
    \ar@<0.5ex>[l]
  }\!\!$}}
from \exaref{1}. Under the given assumptions on $T$, it is straightforward to verify that this is a tilting adjunction in the sense of \dfnref{tilting-adjunction}. Now apply \thmref{equivalence-simplified}.
\end{proof}

Recall that a \emph{semidualizing} module over a commutative noetherian ring $R$ is nothing but a (balanced) Wakamatsu tilting module for the pair $(R,R)$.

The next consequence of \thmref{equivalence} seems to be new in the case where $\ell>0$. For $\ell=0$ it is a classic result, sometimes called \emph{Foxby equivalence}, of Foxby \cite[Sect.~1]{HBF72}; see also Avramov and Foxby \cite[Thm.~(3.2) and Prop.~(3.4)]{LLAHBF97} and Christensen \cite[Obs.~(4.10)]{Christensen01}.

\begin{cor}[Foxby]
  \label{cor:Foxby}
  Let $R$ be a commutative noetherian ring. If $C$ is a semidualizing $R$-module, then there is for every $\ell \in \mathbb{Z}$ an adjoint equivalence:
\begin{align*}
  \xymatrix@R=2pc{
  {
  \left\{ \!\textnormal{\normalsize $M \in \Mod{R}$} \,\left|\!
    \begin{array}{l}
      \textnormal{\normalsize $\Tor{R}{i}{C}{M}=0$ for all $i \neq \ell$,} \\
      \textnormal{\normalsize $\Ext{R}{i}{C}{\Tor{R}{\ell}{C}{M}}=0$ for all $i \neq \ell$, and} \\      
      \textnormal{\normalsize $\eta^\ell_M \colon M \to \Ext{R}{\ell}{C}{\Tor{R}{\ell}{C}{M}}$ is an isomorphism}
    \end{array}
  \right. \!\!\!\right\}\phantom{.}
  }
  \ar@<-0.8ex>[d]_-{\Tor{R}{\ell}{C}{-}}
  \\
  {
  \left\{ \!\textnormal{\normalsize $N \in \Mod{R}$} \,\left|\!
    \begin{array}{l}
      \textnormal{\normalsize $\Ext{R}{i}{C}{N}=0$ for all $i \neq \ell$,} \\
      \textnormal{\normalsize $\Tor{R}{i}{C}{\Ext{R}{\ell}{C}{N}}=0$ for all $i \neq \ell$, and} \\      
      \textnormal{\normalsize $\varepsilon^\ell_N \colon \Tor{R}{\ell}{C}{\Ext{R}{\ell}{C}{N}} \to B$ is an isomorphism}
    \end{array}
  \right. \!\!\!\right\}.
  }
  \ar@<-0.8ex>[u]_-{\Ext{R}{\ell}{C}{-}}
  }
\end{align*}      
\end{cor}

\begin{proof}
  Apply \thmref{equivalence} to \exaref{2} with $\Gamma=R=\Lambda$ and $T=C$.
\end{proof}

\begin{exa}
  \label{exa:Matlis1}
  Let $(R,\mathfrak{m},k)$ be a commutative noetherian local ring. Recall that an $R$-module $M$ is \emph{Matlis reflexive} if the canonical map $M \to \Hom{R}{\Hom{R}{M}{E_R(k)}}{E_R(k)}$ is an isomorphism. By applying \thmref{equivalence} with $\ell=0$ to the adjunction from \exaref{1} with $\Gamma=R=\Lambda$ and $T=E_R(k)$, one gets the (almost trivial) adjoint equivalence:
\begin{displaymath} 
  \xymatrix@C=6pc{
    \{ \textnormal{Matlis reflexive $R$-modules} \}
    \ar@<0.6ex>[r]^-{\Hom{R}{-}{E_R(k)}^\mathrm{op}} 
    & 
    \{ \textnormal{Matlis reflexive $R$-modules} \}^{\mathrm{op}}
    \ar@<0.6ex>[l]^-{\Hom{R}{-}{E_R(k)}}    
  }\!.
\end{displaymath}
\end{exa}

\section{Derivatives of the main result in the case $\ell=0$}
\label{sec:Derivatives}

In this section, we consider the equivalence from \thmref{equivalence} with $\ell=0$ and show~that sometimes it restricts to an equivalence between certain ``finite'' objects in $\Fix[]{0}{\cA}$ and $\coFix[]{0}{\cB}$. The precise statements can be found in \thmref[Theorems~]{equivalence-gen} and \thmref[]{equivalence-res}. 

For an object $X$ in an abelian category $\cC$ we use the standard notation $\add[\cC]{X}$ for the class of objects in $\cC$ that are direct summands in finite direct sums of copies of $X$. 

\begin{dfn}
  \label{dfn:gen-res}
  Let $\cC$ be an abelian category, let $X \in \cC$, and let $d \in \mathbb{N}_0$. 
  
  An object \mbox{$C \in \cC$} is said to be \emph{$d$-generated by $X$}, respectively, \emph{$d$-cogenerated by $X$}, if there is an exact sequence $X_d \to \cdots \to X_0 \to C \to 0$, respectively, $0 \to C \to X^0 \to \cdots \to X^d$, where $X_0,\ldots,X_d$, respectively, $X^0,\ldots,X^d$, belong to $\add[\cC]{X}$. The full subcategory of $\cC$ consisting of all such objects is denoted by $\gen{\cC}{d}{X}$, respectively, $\cogen{\cC}{d}{X}$.
  
  We say that \mbox{$C \in \cC$} \emph{has an $\add[\cC]{X}$-resolution of length $d$}, respectively, \emph{has an $\add[\cC]{X}$-coresolution of length $d$}, if there exists an exact sequence $0 \to X_d \to \cdots \to X_0 \to C \to 0$, respectively, $0 \to C \to X^0 \to \cdots \to X^d \to 0$, where $X_0,\ldots,X_d$, respectively, $X^0,\ldots,X^d$, belong to $\add[\cC]{X}$. The full subcategory of $\cC$ consisting of all such objects is denoted by $\res{\cC}{d}{X}$, respectively, $\cores{\cC}{d}{X}$.  
\end{dfn}

\begin{rmk}
  \label{rmk:gen-op}
  Note that as full subcategories of $\cC^\mathrm{op}$ one has $\gen{\cC^\mathrm{op}}{d}{X} = \cogen{\cC}{d}{X}^\mathrm{op}$ and $\res{\cC^\mathrm{op}}{d}{X} = \cores{\cC}{d}{X}^\mathrm{op}$.
  Also note that $\res{\cC}{0}{X} = \add[\cC]{X} = \cores{\cC}{0}{X}$.
\end{rmk}

\begin{exa}
  \label{exa:fg-Artinian}
  Let $(R,\mathfrak{m},k)$ be a commutative noetherian local ring. There are equalities: 
  \begin{align*}
    \gen{\Mod{R}}{0}{R} &= \{ \textnormal{Finitely generated $R$-modules} \}
    \\
    \cogen{\Mod{R}}{0}{E_R(k)} &= \{ \textnormal{Artinian $R$-modules} \}\;,    
  \end{align*}
where the first one is trivial and the second one is well-known; see e.g.~\cite[Thm.~3.4.3]{rha}.

  If $R$ is Cohen--Macaulay with dimension $d$ and a dualizing module $\upOmega$, then one has:
  \begin{align*}
    \res{\Mod{R}}{d}{R} &= \{ \textnormal{Finitely generated $R$-modules with finite projective dimension} \}
    \\
    \res{\Mod{R}}{d}{\upOmega} &= \{ \textnormal{Finitely generated $R$-modules with finite injective dimension} \}\;.    
  \end{align*}
Here the first equality is well-known and the second one follows easily from the existence of maximal Cohen--Macaulay approximations \cite[Thm.~A]{MAsROB89}; see also \cite[Exer.~3.3.28]{BruHer}.
\end{exa}

\begin{lem}
  \label{lem:thick-1}
  For $\ell=0$ the categories from \dfnref{Fix-Cofix} have the following properties:
  \begin{prt}
  \item The category $\Fix[]{0}{\cA}$ is closed under direct summands, extensions, and kernels of epimorphisms in $\cA$. 
  
  \item The category $\coFix[]{0}{\cB}$ is closed under direct summands, extensions, and cokernels of monomorphisms in $\cB$. 
  \end{prt}
\end{lem}

\begin{proof}
  The closure under direct summands and extensions follows from \lemref{extension}. The remaining assertions are proved by using similar methods.
\end{proof}

\begin{lem}
  \label{lem:thick-2}
  For $\ell=0$ the categories from \dfnref{Fix-Cofix} have the following properties:
  \begin{prt}
  \item If the kernel of $\Fu{G}$ is trivial, that is, if $\,\Fu{G}(B)=0$ implies $B=0$ (for any $B \in \cB$), then $\Fix[]{0}{\cA}$ is closed under cokernels of monomorphisms in $\cA$.
  
  \item If the kernel of $\Fu{F}$ is trivial, that is, if $\,\Fu{F}(A)=0$ implies $A=0$ (for any $A \in \cA$), then $\coFix[]{0}{\cB}$ is closed under kernels of epimorphisms in $\cB$.
  \end{prt}
\end{lem}

\begin{proof}
  \proofoftag{a} Let $0 \to A' \to A \to A'' \to 0$ be a short exact sequence in $\cA$ with $A', A \in \Fix[]{0}{\cA}$. Since $\Ld{1}{\Fu{F}}(A)=0$ we obtain the exact sequence \mbox{$0 \to \Ld{1}{\Fu{F}}(A'') \to \Fu{F}(A') \to \Fu{F}(A)$}, and as $\Fu{G}$ is left exact we also get exactness of the sequence $0 \to \Fu{G}(\Ld{1}{\Fu{F}}(A'')) \to \Fu{G}\Fu{F}(A') \to \Fu{G}\Fu{F}(A)$. Since $\eta_{A'}$ and $\eta_{A}$ are isomorphisms, the morphism $\Fu{G}\Fu{F}(A') \to \Fu{G}\Fu{F}(A)$ may be identified with $A' \to A$, which is mono. It follows~that $\Fu{G}(\Ld{1}{\Fu{F}}(A''))=0$, and consequently $\Ld{1}{\Fu{F}}(A'')=0$. Having established this, arguments as in the proof of \lemref{extension} show that $A'' \in \Fix[]{0}{\cA}$.
  
  \proofoftag{b} Similar to the proof of part (a). 
\end{proof}

We give a few examples of adjunctions that satisfy the hypotheses in \lemref{thick-2}.

\begin{exa}
  \label{exa:faithful-adjunction-2}
  Let $R$ be a commutative ring and let $E$ be a faithfully injective $R$-module, that is, the functor $\Hom{R}{-}{E}$ is faithfully exact. In this case, the adjunction $(\Fu{F},\Fu{G}) = (\Hom{R}{-}{E}^\mathrm{op},\Hom{R}{-}{E})$ from \exaref{1} has the property that either of the conditions $\Fu{F}(M)=0$ or $\Fu{G}(M)=0$ imply $M=0$ (for any $R$-module $M$).
\end{exa}

\begin{exa}
  \label{exa:faithful-adjunction-1}
  Let $R$ be a commutative noetherian ring and let $C$ be a finitely generated $R$-module with $\mathrm{Supp}_RC=\mathrm{Spec}\,R$. In this case, the adjunction $(\Fu{F},\Fu{G}) = (C\otimes_R-,\Hom{R}{C}{-})$ from \exaref{2} has the property that either of the conditions $\Fu{F}(M)=0$ or $\Fu{G}(M)=0$ imply $M=0$ (for any $R$-module $M$).  This follows from basic results in commutative algebra; cf.~\cite[\S3.3]{HHlDWh08}.
\end{exa}

\begin{thm}
  \label{thm:equivalence-gen}
  Assume that $\Fu{F}(A)=0$ implies $A=0$ (for any $A \in \cA$). For any $X \in \Fix[]{0}{\cA}$ and $d \geqslant 0$ the equivalence from \thmref{equivalence} with $\ell=0$ restricts to an equivalence:
\begin{displaymath}
  \xymatrix@C=3pc{
    \Fix[]{0}{\cA} \mspace{1mu}\cap\mspace{1mu} \gen{\cA}{d}{X}\, \ar@<0.6ex>[r]^-{\Fu{F}} & \, \coFix[]{0}{\cB} \mspace{1mu}\cap\mspace{1mu} \gen{\cB}{d}{\Fu{F}X} \ar@<0.6ex>[l]^-{\Fu{G}}
  }\!.
\end{displaymath}  
\end{thm}

\begin{proof}
  In view of \thmref{equivalence} we only have to argue that $\Fu{F}$ maps $\Fix[]{0}{\cA} \cap \gen{\cA}{d}{X}$ to $\gen{\cB}{d}{\Fu{F}X}$ and that $\Fu{G}$ maps $\coFix[]{0}{\cB} \cap \gen{\cB}{d}{\Fu{F}X}$ to $\gen{\cA}{d}{X}$. 
  
  First assume that $A$ belongs to $\Fix[]{0}{\cA} \cap \gen{\cA}{d}{X}$. Since $A \in \gen{\cA}{d}{X}$ there is an exact sequence $X_d \to \cdots \to X_0 \to A \to 0$ with $X_0,\ldots,X_d \in \add[\cA]{X}$. Since $A, X \in \Fix[]{0}{\cA}$ one has, in particular, $\Ld{i}{\Fu{F}}(A)=0=\Ld{i}{\Fu{F}}(X_n)$ for all $i>0$ and $n=0,\ldots,d$, so it follows that the sequence $\Fu{F}X_d \to \cdots \to \Fu{F}X_0 \to \Fu{F}A \to 0$ is exact, and hence $\Fu{F}A$ belongs to $\gen{\cB}{d}{\Fu{F}X}$.
  
   Next assume that $B$ is in $\coFix[]{0}{\cB} \cap \gen{\cB}{d}{\Fu{F}X}$ and let 
   $Y_d \to \cdots \to Y_0 \to B \to 0$ be an exact sequence in $\cB$ with $Y_0,\ldots,Y_d \in \add[\cB]{\Fu{F}X}$. As $X \in \Fix[]{0}{\cA}$ we have $\Fu{F}X \in \coFix[]{0}{\cB}$ and hence $Y_0,\ldots,Y_d \in \coFix[]{0}{\cB}$. The assumption on $\Fu{F}$ and \lemref{thick-2}(b) imply that $\coFix[]{0}{\cB}$ is closed under kernels of epi\-mor\-phisms in $\cB$, and consequently all the kernels 
$K_0 = \Ker{(Y_0 \twoheadrightarrow B)},\, K_1 = \Ker{(Y_1 \twoheadrightarrow K_0)}, \ldots,\, K_d = \Ker{(Y_d \twoheadrightarrow K_{d-1})}$ belong to $\coFix[]{0}{\cB}$. In particular, one has $\Rd{i}{\Fu{G}}(K_0)=\Rd{i}{\Fu{G}}(K_1)=\cdots = \Rd{i}{\Fu{G}}(K_d)=0$ for all $i>0$, and hence the sequence $\Fu{G}Y_d \to \cdots \to \Fu{G}Y_0 \to \Fu{G}B \to 0$ is exact. As $\Fu{G}\Fu{F}X \cong X$ and $Y_0,\ldots,Y_d \in \add[\cB]{\Fu{F}X}$, it follows that $\Fu{G}Y_0,\ldots,\Fu{G}Y_d \in \add[\cA]{X}$, and thus $\Fu{G}B \in \gen{\cA}{d}{X}$. 
\end{proof}

\enlargethispage{5.8ex}

The following corollary of \thmref{equivalence-gen} is a classic result of Matlis~\cite[Cor.~4.3]{EMt58}.

\begin{cor}[Matlis]
  \label{cor:Matlis}
   Let $(R,\mathfrak{m},k)$ be a commutative noetherian local $\mathfrak{m}$-adically complete ring. There is an adjoint equivalence:
\begin{displaymath} 
  \xymatrix@C=6pc{
    \{ \textnormal{Finitely generated $R$-modules} \}
    \ar@<0.6ex>[r]^-{\Hom{R}{-}{E_R(k)}^\mathrm{op}} 
    & 
    \{ \textnormal{Artinian $R$-modules} \}^{\mathrm{op}}
    \ar@<0.6ex>[l]^-{\Hom{R}{-}{E_R(k)}}    
  }\!.
\end{displaymath}
\end{cor}

\begin{proof}
  Consider the situation from \exaref{Matlis1}. The assumption that $R$ is $\mathfrak{m}$-adically complete yields that $R$ (viewed as an $R$-module) is Matlis reflexive; see e.g.~\cite[Thm.~3.4.1(8)]{rha}. The asserted equivalence now follows directly from \thmref{equivalence-gen} with $X=R$ and $d=0$ in view of \exaref{faithful-adjunction-2} and of \rmkref{gen-op} and \exaref{fg-Artinian} (first half).
\end{proof}

\begin{thm}
  \label{thm:equivalence-res}
 For any $X \in \Fix[]{0}{\cA}$ and $d \geqslant 0$ the equivalence from \thmref{equivalence} with $\ell=0$ restricts to an equivalence:
\begin{displaymath}
  \xymatrix@C=3pc{
    \Fix[]{0}{\cA} \mspace{1mu}\cap\mspace{1mu} \res{\cA}{d}{X}\, \ar@<0.6ex>[r]^-{\Fu{F}} & \, \res{\cB}{d}{\Fu{F}X} \ar@<0.6ex>[l]^-{\Fu{G}}
  }\!.
\end{displaymath}  
If $\,\Fu{G}(B)=0$ implies $B=0$ (for any $B \in \cB$), then $\res{\cA}{d}{X} \subseteq \Fix[]{0}{\cA}$ and hence the equivalence takes the simpler form $\res{\cA}{d}{X} \leftrightarrows \res{\cB}{d}{\Fu{F}X}$.
\end{thm}

\begin{proof}
  By \lemref{thick-1}(b) the class $\coFix[]{0}{\cB}$ is closed under cokernels of monomorphisms in $\cB$, and therefore $\res{\cB}{d}{\Fu{F}X} \subseteq \coFix[]{0}{\cB}$. So in view of \thmref{equivalence} we only have to show that $\Fu{F}$ maps $\Fix[]{0}{\cA} \cap \res{\cA}{d}{X}$ to $\res{\cB}{d}{\Fu{F}X}$ and that $\Fu{G}$ maps $\res{\cB}{d}{\Fu{F}X}$~to~$\res{\cA}{d}{X}$. This follows from arguments similar to the ones found in the proof of \thmref{equivalence-gen}. The last assertion follows from \lemref{thick-2}(a).
\end{proof}

The following corollary of \thmref{equivalence-res} is a classic result of Sharp \cite[Thm.~(2.9)]{RYS72}.

\begin{cor}[Sharp]
  \label{cor:Sharp}
   Let $(R,\mathfrak{m},k)$ be a commutative noetherian local Cohen--Macaulay ring with a dualizing module $\upOmega$. There is an adjoint equivalence:
\begin{displaymath}
  \xymatrix@C=5pc{
  {
  \left\{\!\!\!\! 
    \begin{array}{c}
      \textnormal{\normalsize Finitely generated $R$-modules} \\
      \textnormal{\normalsize with finite projective dimension} 
    \end{array}
  \!\!\!\!\right\}
  }
  \ar@<0.8ex>[r]^-{\upOmega\,\otimes_R\,-}
  &
  {
  \left\{\!\!\!\! 
    \begin{array}{c}
      \textnormal{\normalsize Finitely generated $R$-modules} \\
      \textnormal{\normalsize with finite injective dimension} 
    \end{array}
  \!\!\!\!\right\}
  }  \ar@<0.8ex>[l]^-{\Hom{R}{\upOmega}{-}}
  }\!.
\end{displaymath}            
\end{cor}

\begin{proof}
  Immediate from \exaref{faithful-adjunction-1}, \thmref{equivalence-res} with $X=R$, and \exaref{fg-Artinian}.
\end{proof}

\section{Applications to relative Cohen--Macaulay modules}
\label{sec:CM}

Throughout this section, $(R,\mathfrak{m},k)$ is a commutative noetherian local ring and $\mathfrak{a} \subset R$ is a proper ideal. We apply \thmref{equivalence} to study the category of (not necessarily finitely generated) relative Cohen--Macaulay modules. Our main result is \thmref{CM2}.

We begin by recalling a few well-known defintions and facts about local (co)homology. 

\begin{ipg}
\label{lc-definitions}
The \emph{$\mathfrak{a}$-torsion functor} and the \emph{$\mathfrak{a}$-adic completion  functor} are defined by 
\begin{displaymath}
  \G{} = \textstyle\varinjlim_{n \in \mathbb{N}} \Hom{R}{R/\mathfrak{a}^n}{-}
  \quad \text{and} \quad
  \L{} = \varprojlim_{n \in \mathbb{N}}(R/\mathfrak{a}^n \otimes_R -)\;.
\end{displaymath}
The $i^\mathrm{th}$ right derived functor of $\G{}$ is written $\lce{i}$ and called the \emph{$i^\mathrm{th}$ local cohomology w.r.t.~$\mathfrak{a}$}. The $i^\mathrm{th}$ left derived functor of $\L{}$ is written $\lhe{i}$ and called the \emph{$i^\mathrm{th}$ local homology w.r.t.~$\mathfrak{a}$}.

  The functor $\L{}$ is \textsl{not} right exact on the category of all $R$-modules, so its zeroth left derived functor $\lhe{0}$ is in general \textsl{not} naturally isomorphic to $\L{}$. For every $R$-module $M$ there are canonical homomorphisms $\psi_M \colon M \to \lh{0}{M}$ and $\varphi_M \colon \lh{0}{M} \twoheadrightarrow \L{M}$ whose composite $\varphi_M\circ\psi_M$ is the $\mathfrak{a}$-adic completion map \mbox{$\tau_M \colon M \to \L{M}$}; see Simon \cite[\S5.1]{AMS90}. On the category of finitely generated $R$-modules, the functor $\L{}$ is exact, as it is naturally isomorphic to  \smash{$-\otimes_R\widehat{R}^{\mathfrak{a}}$}; see \cite[Thms.~8.7 and 8.8]{Mat}. Hence, if $M$ is a finitely generated $R$-module, $\varphi_M$ is an isomorphism, $\psi_M$ may be identified with $\tau_M$, and $\lh{i}{M}=0$ for $i>0$.

On the derived category $\D{R}$ one can consider the total right derived functor $\RG$~of~$\G{}$. A classic result due to Grothendieck \cite[Prop.~1.4.1]{EGA3}\footnote{\,See also Brodmann and Sharp \cite[Thm.~5.1.19]{BroSha}, Alonso Tarr{\'{\i}}o, Jerem{\'{\i}}as L{\'o}pez, and Lipman \cite[Lem.~3.1.1]{AJL-97} (with corrections by Schenzel \cite{PSc03}), and Porta, Shaul, and Yekutieli \cite[Prop.~5.8]{PSY}.} asserts that $\RG \cong \C \otimes^\mathbf{L}_R-$, where $\C$ is the {\v C}ech complex on any set of generators of $\mathfrak{a}$. Similarly, $\LL \cong \RHom{R}{\C}{-}$ by Greenlees and
May \cite[Sect.~2]{JPGJPM92} (with corrections by Schenzel \cite{PSc03})\footnote{\,See also  Porta, Shaul, and Yekutieli \cite[Cor.~7.13]{PSY} for a very clear exposition.}. For any $R$-module $M$ one has by definition $\lc{i}{M} = \mathrm{H}_{-i}(\RG{M})$ and $\lh{i}{M} = \mathrm{H}_{i}(\LL{M})$.
\end{ipg}
  
\begin{ipg}
  \label{lc-facts}
Recall that for any $R$-module $M$, its \emph{depth (or grade) w.r.t.~$\mathfrak{a}$} is the number 
\begin{displaymath}
  \mathrm{depth}_R(\mathfrak{a},M) = \inf\{i \,|\, \Ext{R}{i}{R/\mathfrak{a}}{M} \neq 0\} \in \mathbb{N}_0 \cup \{\infty\}\;.
\end{displaymath}
If $M$ is finitely generated, then this number is the common length all maximal $M$-sequences contained in $\mathfrak{a}$; see \cite[\S1.2]{BruHer}. Strooker \cite[Prop.~5.3.15]{Strooker} shows that for every $M$ one has:
  \begin{displaymath}
     \inf\{i\,|\,\lc{i}{M}\neq 0\} = \mathrm{depth}_R(\mathfrak{a},M)\;.
  \end{displaymath}  
  Thus, if $M$ is finitely generated, then $\inf\{i\,|\,\lc[\mathfrak{m}]{i}{M}\neq 0\} = \mathrm{depth}_RM$.
 
The number $\sup\{i\,|\,\lc{i}{M}\neq 0\}$ is less well understood; it is often called the \emph{cohomological dimension} of $M$ w.r.t.~$\mathfrak{a}$ and denoted by $\mathrm{cd}_R(\mathfrak{a},M)$. If $M$ is finitely generated, then 
$\mathrm{cd}_R(\mathfrak{m},M) = \mathrm{dim}_RM$ by \cite[Thms.~6.1.2 and 6.1.4]{BroSha}.
\end{ipg}

From \ref{lc-facts} one gets the well-known fact that a (non-zero) finitely generated module $M$ is Cohen--Macaulay with $t=\mathrm{depth}_RM=\mathrm{dim}_RM$ if and only if $\lc[\mathfrak{m}]{i}{M}=0$ for all $i \neq t$. In view of this, the next definition due to Zargar \cite[Def.~2.1]{Zargar} is natural.

\begin{dfn}[Zargar]
  \label{dfn:Zargar}
A finitely generated $R$-module $M$ is said to be \emph{relative Cohen--Macaulay of cohomological dimension $t$ w.r.t.~$\mathfrak{a}$} if one has if $\lc{i}{M}=0$ for all $i \neq t$. 
\end{dfn}

The ring $R$ is said to be \emph{relative Cohen--Macaulay w.r.t.~$\mathfrak{a}$} if it is so when viewed as a module over itself, that is, if 
$c(\mathfrak{a}):=\mathrm{grade}_R(\mathfrak{a},R) = \mathrm{cd}_R(\mathfrak{a},R)$. In the terminology of Hellus and Schenzel \cite{HellusSchenzel}, this means that $\mathfrak{a}$ is a \emph{cohomologically complete intersection ideal}. 

\begin{exa}
  Let $x_1,\ldots,x_n \in R$ be a sequence of elements. It follows from \cite[Thm.~3.3.1]{BroSha} (and \ref{lc-facts}) that any finitely generated $R$-module $M$ for which $x_1,\ldots,x_n$ is an $M$-sequence is relative Cohen--Macaulay of cohomological dimension $n$ with respect to $\mathfrak{a}=(x_1,\ldots,x_n)$. In particular, if $x_1,\ldots,x_n$ is an $R$-sequence, then $R$ is relative Cohen--Macaulay with respect to $\mathfrak{a}=(x_1,\ldots,x_n)$ and one has $c(\mathfrak{a})=n$.
\end{exa}

For a ring $R$ that is relative Cohen--Macaulay w.r.t.~$\mathfrak{a}$ we now set out to study the category
\begin{displaymath}
  \{M \in \mod{R} \,|\, \lc{i}{M}=0 \textnormal{ for all } i \neq t \}
  \quad \textnormal{(for any $t$)}
\end{displaymath}
of finitely generated relative Cohen--Macaulay of cohomological dimension $t$ w.r.t.~$\mathfrak{a}$. But first we extend the notion of relative Cohen--Macaulayness to the realm of all modules. 

\begin{dfn}
  An $R$-module $M$ is said to be \emph{$\mathfrak{a}$-trivial} if $\lh{i}{M}=0$ for all $i \in \mathbb{Z}$.
\end{dfn}

\begin{rmk}
  By Strooker \cite[Prop.~5.3.15]{Strooker} and Simon \cite[Thm.~2.4 and Cor.~p.~970 part (ii)]{Simon92} $\mathfrak{a}$-trivialness of a module $M$ is equivalent to any of the conditions: (i) $\lc{i}{M}=0$ for all $i \in \mathbb{Z}$. \ (ii) $\Ext{R}{i}{R/\mathfrak{a}}{M}=0$ for all $i \in \mathbb{Z}$. \ (iii) $\Tor{R}{i}{R/\mathfrak{a}}{M}=0$ for all $i \in \mathbb{Z}$.
\end{rmk}

We denote the Matlis duality functor $\Hom{R}{-}{E_R(k)}$ by $(-)^v$, and for an $R$-module $M$ we write $\delta_M \colon M \to M^{vv}$ for the canonical monomorphism. 

\begin{dfn}
  \label{dfn:CM}
An $R$-module $M$ (not necessarily finitely generated) is said to be \emph{relative Cohen--Macaulay of cohomological dimension $t$ w.r.t.~$\mathfrak{a}$} if it satisfies the conditions:
\begin{prt}
\item[\hspace{2ex} \con{CM1}] $\lc{i}{M}=0$ for all $i \neq t$.
\item[\hspace{2ex} \con{CM2}] The canonical map $\psi_M \colon M \to \lh{0}{M}$ is an isomorphism.
\item[\hspace{2ex} \con{CM3}] The cokernel of $\delta_M \colon M \to M^{vv}$ is $\mathfrak{a}$-trivial.
\end{prt}
The category of all such $R$-modules is denoted $\CM{t}{R}$. 
\end{dfn}

\begin{obs}
  \label{obs:fg}
  Assume that $R$ is $\mathfrak{m}$-adically complete, and hence also $\mathfrak{a}$-adically complete by \cite[Cor.~2.2.6]{Strooker}. In this case, conditions \con{CM2} and \con{CM3} automatically hold for all finitely generated $R$-modules, see \ref{lc-definitions} and \cite[Thm.~3.4.1(8)]{rha}, so there is an equality,
\begin{displaymath}
  \CM{t}{R} \,\cap\, \mod{R} \,=\, \{M \in \mod{R} \,|\, \lc{i}{M}=0 \textnormal{ for all } i \neq t \}\;.
\end{displaymath}
Thus, in this case, a finitely generated module is relative Cohen--Macaulay w.r.t.~$\mathfrak{a}$ in the sense of \dfnref{CM} if and only if it is so in the sense of Zargar (\dfnref{Zargar}).
\end{obs}

\begin{exa}
  \label{exa:CM-0}
  For $\mathfrak{a}=0$ we have $\G{} = \Fu{Id}_{\Mod{R}}=\L{}$, and the only $\mathfrak{a}$-trivial module is the zero module. Thus, for $\mathfrak{a}=0$ one has \smash{$\CM{0}{R} =  \{ \textnormal{Matlis reflexive $R$-modules} \}$}.
\end{exa}

\begin{lem}
\label{lem:Hc}
Assume that $R$ is relative Cohen--Macaulay w.r.t.~$\mathfrak{a}$ in the sense of \dfnref{Zargar} and set $c=c(\mathfrak{a})$. In this case, the $R$-module $\lc{c}{R}$ has the following properties:
\begin{prt}   
\item $\lc{c}{R}$ has finite projective dimension.
\item $\Ext{R}{i}{\lc{c}{R}}{\lc{c}{R}}=0$ for all $i>0$.
\item $\Hom{R}{\lc{c}{R}}{\lc{c}{R}}$ is isomorphic to the $\mathfrak{a}$-adic completion \smash{$\widehat{R}^{\mathfrak{a}}$}.
\item There are isomorphisms \smash{$\RG{} \cong \upSigma^{-c}(\lc{c}{R} \otimes^\mathbf{L}_R-)$} and \smash{$\lce{i} \cong \Tor{R}{c-i}{\lc{c}{R}}{-}$}.
\item There are isomorphisms \smash{$\LL{} \cong \upSigma^{c}\RHom{R}{\lc{c}{R}}{-}$} and \smash{$\lhe{i} \cong \Ext{R}{c-i}{\lc{c}{R}}{-}$}.
\end{prt}
\end{lem}

\begin{proof}
  Since $\lc{i}{R} \cong \mathrm{H}_{-i}(\RG{(R)}) \cong \mathrm{H}_{-i}(\C)$ by \ref{lc-definitions}, the assumption that $R$ is relative Cohen--Macaulay w.r.t.~$\mathfrak{a}$ means that the homology of $\C$ is concentrated in degree $-c$. Thus there are isomorphisms $\lc{c}{R} \cong \operatorname{H}_{-c}(\C) \cong \upSigma^{c}\C$ in $\D{R}$. In view of this, part (a) follows since $\C$ has finite projective dimension, see \cite[\S5.8]{CFH-06}, parts (b) and (c) follow from \cite[Lem.~1.9]{AFrPJr04}, and (d) and (e) follow from~\ref{lc-definitions}.~\qedhere
\end{proof}

\begin{dfn}
  \label{dfn:theta}
Naturality of $\psi$ from \ref{lc-definitions} shows that 
for any $R$-module $M$ there is an equality $\psi_{M^{vv}} \circ \delta_M = \lh{0}{\delta_M} \circ \psi_M$ of homomorphisms $M \to \lh{0}{M^{vv}}$; we write $\theta_M$ for this map.
\end{dfn}

\begin{lem}
  \label{lem:theta}
  An $R$-module $M$ satisfies \con{CM2} and \con{CM3} in \dfnref{CM} if and only if
  \begin{prt}
  \item[\textnormal{($\dagger$)}] $\lh{i}{M^{vv}}=0$ for all $i>0$, and
  \item[\textnormal{($\ddagger$)}] $\theta_M \colon M \to \lh{0}{M^{vv}}$ is an isomorphism.
  \end{prt}  
\end{lem}

\begin{proof}
  ``Only if'': By \con{CM2} and \cite[p.~238, second Lem., part (ii)]{AMS90} we get isomorphisms $\lh{i}{M} \cong \lh{i}{\lh{0}{M}} = 0$ for all $i>0$. The exact sequence \mbox{$0 \to M \to M^{vv} \to C_M \to 0$}, where the map from $M$ to $M^{vv}$ is $\delta_M$ and $C_M=\Coker{\delta_M}$, induces a long exact sequence of local homology modules w.r.t~$\mathfrak{a}$, and since $C_M$ is  $\mathfrak{a}$-trivial by \con{CM3}, we conclude that $\lh{i}{\delta_M} \colon \lh{i}{M} \to \lh{i}{M^{vv}}$ is an isomorphism for all $i \in \mathbb{Z}$. Thus ($\dagger$) follows. As $\lh{0}{\delta_M}$ is an isomorphism, so is $\theta_M=\lh{0}{\delta_M} \circ \psi_M$, that is, ($\ddagger$) holds. 

``If'': As ($\ddagger$) holds, $M$ has the form $M \cong \lh{0}{X}$ so \cite[p.~238, second Lem., part~(ii)]{AMS90} yields that $\psi_M \colon M \to \lh{0}{M}$ is an isomorphism, i.e.~\con{CM2} holds, and $\lh{i}{M}=0$ for~$i>0$. As $\theta_M=\lh{0}{\delta_M} \circ \psi_M$ and $\psi_M$ are both isomorphisms, so is $\lh{0}{\delta_M}$. By ($\dagger$) we also have $\lh{i}{M^{vv}}=0$ for all $i>0$, so the long exact sequence of local homology modules induced by \mbox{$0 \to M \to M^{vv} \to C_M \to 0$} shows that $\lh{i}{C_M}=0$ for all $i \in \mathbb{Z}$, i.e.~\con{CM3} holds.
\end{proof}

We prove in \thmref{CM2} below that the category $\CM{t}{R}$ is self-dual. The duality~is realized via the following module which was already introduced by Zargar \cite[Def.~2.3]{ZargarDuality}.

\begin{dfn}[Zargar]
  \label{dfn:Omega}
  Let $R$ be relative Cohen--Macaulay w.r.t.~$\mathfrak{a}$ in the sense of \dfnref{Zargar}. With $c = c(\mathfrak{a})$ we set $\upOmega_\mathfrak{a} = \lc{c}{R}^v = \Hom{R}{\lc{c}{R}}{E_R(k)}$.
\end{dfn}

In the extreme cases $\mathfrak{a}=0$ and $\mathfrak{a}=\mathfrak{m}$ the module $\upOmega_\mathfrak{a}$ is well-understood:

\begin{exa}
  \label{exa:Omega}
  Any ring $R$ is relative Cohen--Macaulay w.r.t.~$\mathfrak{a}=0$; in this case one has $c=0$, $\lc{0}{R}=R$, and $\upOmega_\mathfrak{a} = E_R(k)$.
  
  Assume that $R$ is Cohen--Macaulay (w.r.t.~$\mathfrak{m}$) and $\mathfrak{m}$-adically complete. In this case, one has  $c=\operatorname{depth}\mspace{1mu}R=\operatorname{dim}\mspace{1mu}R$ and $\lc[\mathfrak{m}]{c}{R}$ is Artinian by \cite[Thm.~7.1.3]{BroSha}. Thus $\upOmega_\mathfrak{m} = \lc[\mathfrak{m}]{c}{R}^v$ is finitely generated so \prpref{Omega} below shows that $\upOmega_\mathfrak{m}$ is the dualizing module for $R$.
\end{exa}

\begin{prp}
  \label{prp:Omega}
  If $R$ is $\mathfrak{m}$-adically complete and relative Cohen--Macau\-lay w.r.t.~$\mathfrak{a}$, then $\upOmega_\mathfrak{a}$ has finite injective dimension, $\Ext{R}{i}{\upOmega_\mathfrak{a}}{\upOmega_\mathfrak{a}}=0$ for $i>0$, and $\Hom{R}{\upOmega_\mathfrak{a}}{\upOmega_\mathfrak{a}}\cong R$. Furthermore, in the derived category $\D{R}$ there is an isomorphism
  $\upOmega_\mathfrak{a} \cong \upSigma^{-c}\LL{E_R(k)}$.
\end{prp}

\begin{proof}
  It is immediate from \lemref{Hc}(a) that $\upOmega_\mathfrak{a}$ has finite injective dimension. Part (e) of the same lemma shows that $\upOmega_\mathfrak{a} \cong \upSigma^{-c}\LL{E_R(k)}$ in $\D{R}$, and hence
\begin{displaymath}
  \RHom{R}{\upOmega_\mathfrak{a}}{\upOmega_\mathfrak{a}} 
\cong \RHom{R}{\LL{E_R(k)}}{\LL{E_R(k)}} \cong \LL{\RHom{R}{E_R(k)}{E_R(k)}}\;,
\end{displaymath}
where the last isomorphism comes from \cite[(2.6)]{Frankild} and \cite[Lem.~7.6]{PSY}. As $R$ is $\mathfrak{m}$-adically complete, we have $\RHom{R}{E_R(k)}{E_R(k)} \cong R$, and thus the last expression above is the same as \smash{$\LL{R} \cong \widehat{R}^{\mathfrak{a}}$}. As $R$ is also $\mathfrak{a}$-adically complete, we get $\RHom{R}{\upOmega_\mathfrak{a}}{\upOmega_\mathfrak{a}} \cong R$.
\end{proof}


\begin{thm}
  \label{thm:CM2}
  Assume that $R$ is relative Cohen--Macaulay w.r.t.~$\mathfrak{a}$ in the sense of \dfnref{Zargar} and set $c=c(\mathfrak{a})$. For every integer $t$ there is a duality:
\begin{displaymath}
  \xymatrix@C=5pc{
    \CM{t}{R} \ar@<0.6ex>[r]^-{\Ext{R}{c-t}{-}{\upOmega_\mathfrak{a}}} & \CM{t}{R} \ar@<0.6ex>[l]^-{\Ext{R}{c-t}{-}{\upOmega_\mathfrak{a}}}
  }\!.
\end{displaymath}
\end{thm}

\begin{proof}
We consider the adjunction $(\Fu{F},\Fu{G})$ from \exaref{1} with $\Gamma=R=\Lambda$ and \mbox{$T=\upOmega_\mathfrak{a}$}. From \thmref{equivalence} with $\ell=c-t$ we conclude that the functor $\Ext{R}{c-t}{-}{\upOmega_\mathfrak{a}}$ yields a duality (that is, a ``contravariant equivalence'') on the category $\mathcal{F}:=\Fix[]{c-t}{\Mod{R}}$, whose objects are those $R$-modules $M$ that satisfy the following conditions: (i) $\Ext{R}{i}{M}{\upOmega_\mathfrak{a}}=0$ for all \mbox{$i \neq c-t$}. 
(ii) $\Ext{R}{i}{\Ext{R}{c-t}{M}{\upOmega_\mathfrak{a}}}{\upOmega_\mathfrak{a}}=0$ for all \mbox{$i \neq c-t$}. (iii) The canonical map $\eta^{c-t}_M \colon M \to \Ext{R}{c-t}{\Ext{R}{c-t}{M}{\upOmega_\mathfrak{a}}}{\upOmega_\mathfrak{a}}$ is an isomorphism. We now show $\mathcal{F}=\CM{t}{R}$, that is, we prove that an $R$-module $M$ satisfies (i), (ii), and (iii) if and only if it satisfies \con{CM1}, \con{CM2}, and \con{CM3} in \dfnref{CM}. First note that 
\begin{displaymath}
  \Ext{R}{i}{M}{\upOmega_\mathfrak{a}} = \Ext{R}{i}{M}{\lc{c}{R}^v} \cong \Tor{R}{i}{\lc{c}{R}}{M}^v \cong \lc{c-i}{M}^v\;,
\end{displaymath}
where the last isomorphism is by \lemref{Hc}(d). It follows that condition (i) is equivalent to \con{CM1}. If (i) holds, then $\Ext{R}{c-t}{M}{\upOmega_\mathfrak{a}} \cong \upSigma^{c-t}\RHom{R}{M}{\upOmega_\mathfrak{a}}$ in $\D{R}$, which explains the first isomorphism in the computation below. The second isomorphism below follows as $\upOmega_\mathfrak{a} \cong \upSigma^{-c}\LL{E_R(k)}$, see \prpref{Omega}, and the third isomorphism comes from \cite[(2.6)]{Frankild} and \cite[Lem.~7.6]{PSY}. The last isomorphism is by definition (see \ref{lc-definitions}):
\begin{align*}
  \Ext{R}{i}{\Ext{R}{c-t}{M}{\upOmega_\mathfrak{a}}}{\upOmega_\mathfrak{a}}
  &\cong 
  \mathrm{H}_{-i}\RHom{R}{\upSigma^{c-t}\RHom{R}{M}{\upOmega_\mathfrak{a}}}{\upOmega_\mathfrak{a}}
  \\
  &\cong 
  \mathrm{H}_{(c-t)-i}\RHom{R}{\RHom{R}{M}{\LL{E_R(k)}}}{\LL{E_R(k)}}
  \\        
  &\cong 
  \mathrm{H}_{(c-t)-i}\LL{\RHom{R}{\RHom{R}{M}{E_R(k)}}{E_R(k)}}
  \\
  &\cong 
  \lh{(c-t)-i}{M^{vv}}\;.
\end{align*}
Thus, under assumption of (i), condition (ii) is equivalent to ($\dagger$) $\lh{n}{M^{vv}}=0$ for all $n>0$. Setting $i=c-t$ in the computation above we get an isomorphism,
\begin{displaymath}
  \alpha_M \colon
  \Ext{R}{c-t}{\Ext{R}{c-t}{M}{\upOmega_\mathfrak{a}}}{\upOmega_\mathfrak{a}} \stackrel{\cong}{\longrightarrow} \lh{0}{M^{vv}}\;,
\end{displaymath}
which identifies the map $\eta^{c-t}_M$ from condition (iii) above with the map $\theta_M$ from \dfnref{theta}, that is, $\alpha_M \circ \eta^{c-t}_M = \theta_M$.	So under assumption of (i), condition (iii) is equivalent to ($\ddagger$) $\theta_M$ is an isomorphism. Now apply \lemref{theta}.
\end{proof}

\section*{Acknowledgments}
Part of this work was initiated when Olgur Celikbas visited the Department of Mathematical Sciences at the University of Copenhagen in June 2015. He is grateful for the kind hospitality and the support of the department.

We thank Greg Piepmeyer, Amnon Yekutieli, Majid Rahro Zargar, and the anonymous referee for useful comments and suggestions.

\enlargethispage{6.1ex}

\def\cprime{$'$} \def\soft#1{\leavevmode\setbox0=\hbox{h}\dimen7=\ht0\advance
  \dimen7 by-1ex\relax\if t#1\relax\rlap{\raise.6\dimen7
  \hbox{\kern.3ex\char'47}}#1\relax\else\if T#1\relax
  \rlap{\raise.5\dimen7\hbox{\kern1.3ex\char'47}}#1\relax \else\if
  d#1\relax\rlap{\raise.5\dimen7\hbox{\kern.9ex \char'47}}#1\relax\else\if
  D#1\relax\rlap{\raise.5\dimen7 \hbox{\kern1.4ex\char'47}}#1\relax\else\if
  l#1\relax \rlap{\raise.5\dimen7\hbox{\kern.4ex\char'47}}#1\relax \else\if
  L#1\relax\rlap{\raise.5\dimen7\hbox{\kern.7ex
  \char'47}}#1\relax\else\message{accent \string\soft \space #1 not
  defined!}#1\relax\fi\fi\fi\fi\fi\fi} \def\cprime{$'$}
  \providecommand{\arxiv}[2][AC]{\mbox{\href{http://arxiv.org/abs/#2}{\sf
  arXiv:#2 [math.#1]}}}
  \providecommand{\oldarxiv}[2][AC]{\mbox{\href{http://arxiv.org/abs/math/#2}{\sf
  arXiv:math/#2
  [math.#1]}}}\providecommand{\MR}[1]{\mbox{\href{http://www.ams.org/mathscinet-getitem?mr=#1}{#1}}}
  \renewcommand{\MR}[1]{\mbox{\href{http://www.ams.org/mathscinet-getitem?mr=#1}{#1}}}
\providecommand{\bysame}{\leavevmode\hbox to3em{\hrulefill}\thinspace}
\providecommand{\MR}{\relax\ifhmode\unskip\space\fi MR }
\providecommand{\MRhref}[2]{%
  \href{http://www.ams.org/mathscinet-getitem?mr=#1}{#2}
}
\providecommand{\href}[2]{#2}

\enlargethispage{8ex}

\end{document}